\documentclass[11pt]{article}

\usepackage{graphicx,amsmath,amssymb,amsthm,amsfonts,mathrsfs}
\usepackage[utf8]{inputenc}
\usepackage[mathscr]{euscript}
\usepackage{hyperref}
\usepackage{comment}
\usepackage{xcolor}
\usepackage{float}
\usepackage{epstopdf}
\usepackage{geometry}
\newcommand{\Set}[1]{\left\{#1\right\}}

\newcommand{\Ran}{R}
\newcommand{\Ker}{Ker}

\newcommand{\tm}[1]{\mathrm{#1}}
\newcommand{\skpln}{\vspace*{\baselineskip}}

\newcommand{\bce}{\begin{center}}
\newcommand{\ece}{\end{center}}
\newcommand{\kl}[1]{\left(#1\right)}
\newcommand{\skl}[1]{\left\{ #1 \right\}}
\newcommand{\abs}[1]{\left\vert#1\right\vert}
\newcommand{\inner}[1]{\left\langle#1\right\rangle}
\newcommand{\norm}[1]{\left\| #1 \right\|}

\newcommand{\nalp}{\rho}  % Exponent of the noise decay
\newcommand{\alp}{\alpha}
\newcommand{\gm}{\gamma}

\newcommand{\dlt}{\delta}

\newcommand{\tht}{\theta}

\newcommand{\kp}{\kappa}

\newcommand{\lm}{\lambda}

\newcommand{\vph}{\varphi}

\newcommand{\R}{{\mathbb R}}

\newcommand{\GGm}{\mathbf{G}}

\newcommand{\cv}{\mathbf{c}}

\newcommand{\pv}{\mathbf{p}}

\DeclareMathOperator*{\amin}{argmin}

\numberwithin{equation}{section}

\theoremstyle{plain}
\newtheorem{thm}{Theorem}
\newtheorem{lem}{Lemma}

\newtheorem{corl}{Corollary}
\theoremstyle{definition}

\theoremstyle{remark}
\newtheorem{remk}{Remark}
\newtheorem{assumption}{Assumption}

\makeatletter
\def\moverlay{\mathpalette\mov@rlay}
\def\mov@rlay#1#2{\leavevmode\vtop{%
   \baselineskip\z@skip \lineskiplimit-\maxdimen
   \ialign{\hfil$\m@th#1##$\hfil\cr#2\crcr}}}
\newcommand{\charfusion}[3][\mathord]{
    #1{\ifx#1\mathop\vphantom{#2}\fi
        \mathpalette\mov@rlay{#2\cr#3}
      }
    \ifx#1\mathop\expandafter\displaylimits\fi}
\makeatother

\newcommand*{\rom}[1]{\expandafter\@slowromancap\romannumeral #1@} % roman typographie
\renewcommand{\phi}{\varphi} % nicer phi
\newcommand{\E}{\mathrm{E}} % expectation
\newcommand{\Var}{\mathrm{Var}} % variance
\newcommand{\PP}{\mathrm{P}} % probability
 
\newcommand{\nooutput}[1]{}
\newcommand{\1}{1\!\!\!\!\!\;{\rm I}}

\newcommand{\cF}{{\mathcal F}}

\newcommand{\vf}{\varphi}
\newcommand{\vk}{\varkappa}
\newcommand{\ve}{\varepsilon}
\renewcommand{\lg}{\langle}
\newcommand{\rg}{\rangle}

\newcommand{\const}{\mathop{\rm const}}

\newcommand{\be}{\begin{equation}}
\newcommand{\ee}{\end{equation}}

\title{The quasi-optimality criterion in the linear functional strategy}

\makeatletter
\let\@fnsymbol\@arabic
\makeatother

\author{
Stefan Kindermann\footnote{
Industrial Mathematics Institute,
Johannes Kepler University of Linz,
Altenbergerstrasse 69,
A-4040 Linz, Austria;
{\tt kindermann@indmath.uni-linz.ac.at}
}
\and Sergiy Pereverzyev Jr.\footnote{
Department of Mathematics, University of Innsbruck, Technikerstrasse 13, A-6020 Innsbruck, Austria;
{\tt sergiy.pereverzyev@uibk.ac.at}
}
\and Andrey Pilipenko\footnote{
Institute of Mathematics,  National Academy of Sciences of Ukraine, Tereshchenkivska str. 3, 01601, Kyiv, Ukraine;
Igor Sikorsky Kyiv Polytechnic Institute, Kyiv, Ukraine;
{\tt pilipenko.ay@gmail.com}
}
}

\date{\today}

\begin{document}

\maketitle

\begin{abstract}
The linear functional strategy for the regularization of inverse problems is considered. For selecting
the regularization parameter therein, we propose the heuristic quasi-optimality principle and some 
modifications including the smoothness of  the linear functionals. We prove convergence rates 
for the linear functional strategy with these heuristic rules taking into account the smoothness of 
the solution and the functionals and imposing a structural condition on the noise. Furthermore, 
we study these noise conditions in both a  deterministic and stochastic setup
and verify that for mildly-ill-posed problems and Gaussian noise, these conditions are satisfied 
almost surely, where on the contrary, in the severely-ill-posed case and in a similar setup,  the corresponding noise 
condition fails to hold.  Moreover, we propose an  aggregation method for adaptively optimizing 
the parameter  choice rule by  making  use of improved rates for linear functionals.  
Numerical results indicate that this method yields better results than the standard heuristic rule. 
\end{abstract}

\skpln


\skpln


%%%%%%%%%%%%%%%%%%%%%%%%%%%%%%%%%%%%%%%%%%%%%%%%%%%%%%%%%%%%%%%%%%%%%%%%%%%%%%%%%%%%%%%%%%%%%%%%%%%%%%%%%%%%%%%%%%%%%%%%%%%%%%%%%

\section{Introduction}

%%%%%%%%%%%%%

The estimation of linear bounded functionals of an unknown element $x$ from an indirect noisy observation
$y^{\dlt}$ given as
\begin{equation}\label{n1}
    y^{\dlt} = Tx + \dlt \xi
\end{equation}
is one of the classical problems in regularization theory~\cite{And80}. 
Here, we assume that $T$ is a linear,  injective, not necessarily boundedly invertible operator
 from a solution Hilbert space $X$ into an observation Hilbert space $Y$, $\xi$ is an additive
noise process, and $\dlt$ is its intensity, or noise level, such that for $y=Tx$, it holds
$\norm{  y - y^{\dlt}  } \leq \dlt$, $\dlt\in(0,1)$. We use the same symbols $\inner{ \cdot,\cdot }$, 
$\norm{ \cdot }$
for the inner products and the corresponding norms in both $X$ and~$Y$.

%%%%%%%%%%%%%

It is known that the problem of estimating the value $f(x) = \langle f, x\rangle $ of a linear bounded functional $f\in X$
from~\eqref{n1} is less ill-posed than the problem of estimating $x$, in the sense that the value
$f(x) $ allows for a more accurate reconstruction than the element $x$ in the
$X$-norm~\cite{EngNeu88,AndEng91,LuPer13}.
%%%
%%%
A regularization of the first-named problem is usually performed by the so-called linear functional
strategy~\cite{And86} that is also closely related to the mollifier
methods~\cite{LouMaa90}. In  case of a known noise intensity $\dlt$, the choice of
the regularization parameters in the linear functional strategy has been extensively studied
(see, e.g., \cite{GolPer00,MatPer02,LuPer13} and references therein).

At the same time, in some applications, such as satellite gravity gradiometry, one cannot expect to have
good knowledge of the noise model in general and of the noise intensity $\dlt$ in particular
(see, e.g., discussions in~\cite{KusKle02,BauMatPer07}). As a remedy for this, regularization theory
has an arsenal of so-called heuristic parameter choice strategies that do not require  knowledge
of the noise intensity and therefore can be used in the above mentioned applications.
The quasi-optimality criterion~\cite{TikGla65} is one of the  simplest and the oldest but still quite efficient instance
among such strategies.

%%%%%%%%%%%%%

Of course, in the {\em worst case} scenario, where the noise $\xi$ in~\eqref{n1} is assumed to be chosen 
by some antagonistic
opponent only subject to the constraint $\norm{ \xi } \leq 1$, the quasi-optimality criterion, 
as well as any other heuristic parameter
choice strategy, cannot guarantee convergence of the corresponding regularized approximants because of the so-called
Bakushinskii veto~\cite{Bak84}. On the other hand, it has been shown~\cite{BauRei08,Bec11} that for 
the quasi-optimality
criterion, the Bakushinskii veto can be avoided if the regularization performance is measured {\em on average} over
realizations of $\xi$.

At the same time, another way to overcome the Bakushinskii veto has been proposed in~\cite{KinNeu08,Neu08}, 
where convergence of the regularized approximants to $x$ in the solution space norm and its rates have been
established under a qualitative restriction on the noise $\xi$ (a noise condition of Muckenhoupt type). 
Our intention in this paper
is to extend this {\em restricted noise} approach in~\cite{KinNeu08,Neu08} to the context of the linear functional strategy.
We also show that for a wide class of moderately ill-posed problems~\eqref{n1} and
for random noise $\xi$ with bounded moments, the above mentioned Muckenhoupt-type 
condition is satisfied almost
surely.

The case of severely ill-posed problems is considered as well. Note that in this case, the theoretical bounds
on the convergence rates of the regularized approximants selected by the quasi-optimality criterion in the solution
space norm are  worse than those for the noise level-dependent parameter choice strategies.
At the same time, as follows from our results, in the linear functional strategy, the above-mentioned convergence
rate gap can be essentially reduced. This hints at an opportunity to use the linear functional strategy equipped
with the quasi-optimality criterion for aggregating the constructed regularized approximants in a way described
in~\cite{ChePerXu15}. Then from~\cite{ChePerXu15}, it follows that such aggregation by the linear functional
strategy can improve the accuracy compared to the aggregated regularized approximations, and this
can be seen as a way to use the quasi-optimality criterion for mildly and severely ill-posed problems.

%%%%%%%%%%%%%

Note that a practical implementation of the quasi-optimality criterion depends on the so-called
differential quadrature~\cite{BelKasCas72}
\begin{equation}\label{n2}
  \left.
  \frac{ \partial  x_{\alp}^{\dlt}  }{ \partial  \alp }
  \right| _{ \alp = \alp_i }    \approx
  \sum\limits_{ j }  a_{ij}  x_{\alp_{j}  }^{\dlt}
\end{equation}
that is used to approximate the partial derivative $\frac{ \partial  x_{\alp}^{\dlt}  }{ \partial  \alp }$
of the regularized solution $x_{\alp}^{\dlt}$ of~\eqref{n1}, which is based on a current value
of the regularization parameter $\alp = \alp_i$.
%%%
%%%
Starting from the original paper~\cite{TikGla65}, one usually uses a simple backward difference formula,
where $ a_{ij} = 0 $ for $j\neq i,i-1$, and $a_{ i,i } = - a_{ i,i-1 } = \kl{ \alp_{i} - \alp_{i-1} }^{-1}  $.

On the other hand, as it is mentioned in~\cite{BelKasCas72}, there are many ways of determining the
coefficients $a_{ij}$ in~\eqref{n2}. For example, in the backward difference formula, one can introduce correction
factors such that
$$
  a_{ i,i } =  c_{i}  \kl{ \alp_{i} - \alp_{i-1} }^{-1},\quad
  a_{ i,i-1 } = - c_{i-1} \kl{ \alp_{i} - \alp_{i-1} }^{-1},
$$
where $c_{\ell}$, $\ell = i,i-1$, approximates the values $c_{\ell}^{*}$ minimizing the error
$ \norm{  x - c_{\ell}^{*}  x_{\alp_{ \ell } }^{\dlt}  }  =   \min\limits_{c}  \norm{  x - c  x_{\alp_{ \ell } }^{\dlt}  }  $.
It is clear that
$  c_{\ell}^{*} = \inner{   x,x_{\alp_{ \ell }  }^{\dlt}   }  /  \norm{  x_{\alp_{ \ell }  }^{\dlt}  }^2   $,
and $  \inner{   x,x_{\alp_{ \ell }  }^{\dlt}   }  $ is the value of the linear bounded functional
%%%
$x_{\alp_{ \ell }  }^{\dlt} \in X $
%%%
at the unknown solution that can be approximated by
%%%
$  \inner{   x_{\alp_{ j }  }^{\dlt},     x_{\alp_{ \ell }  }^{\dlt}   }  $,
%%%
where $\alp_j$ is chosen by the quasi-optimality criterion.

%%%%%%%%%%%%%

The use of the backward difference formula corrected as above can be seen as an iterated quasi-optimality rule.
We will demonstrate in Section~\ref{section:numerical} that such a combination of the 
linear functional strategy---by an aggregation approach--- and the quasi-optimality criterion
can also improve the regularization performance as compared to the standard quasi-optimality.

The paper is organized as follows. In the next section, we present the problem setup and formulate main results.
The proofs are given in Section~\ref{section:proof_main}. In Section~\ref{section:sufficient}, 
we describe random processes and investigate whether they  almost surely
meet the Muckenhoupt-type conditions. 
In Section~\ref{section:numerical}, we discuss a combination of the aggregation by means of
the linear functional strategy with the quasi-optimality criterion and present numerical experiments.

%%%%%%%%%%%%%

%%%%%%%%%%%%%%%%%%%%%%%%%%%%%%%%%%%%%%%%%%%%%%%%%%%%%%%%%%%%%%%%%%%%%%%%%%%%%%%%%%%%%%%%%%%%%%%%%%%%%%%%%%%%%%%%%%%%%%%%%%%%%%%%%

\section{The main convergence rates results}\label{sec:two}
In this section, we formulate the main results. Let us introduce some standard notation. 
Let $X,Y$ be Hilbert spaces, $T:X\to Y$ be a continuous linear operator such that $\Ker(T)=\{0\}$, $\Ker(T^*)=\{0\}.$
Here, the assumptions of injectivity of $T$ and $T^*$ are only imposed for simplicity; the main results hold 
with modifications in the general case as well.  We denote by $E_\lambda$ and $F_\lambda$ the 
spectral families for the operators $T^*T$ and $T T^*$, respectively.  The notion 
$\Ran(T)$ stands for the range and $\Ker(T)$ for the nullspace of the operator~$T$.  
For $f,g$ being functions or sequences, the notation $f \asymp g$ indicates that some constants
$c_1,c_2$ exist such that $c_1 f \leq g \leq c_2 f$ for all arguments or sequence indices, where 
the constants in particular do not depend on $\delta$.

Consider an ill-posed problem in the form $Tx=y$. Suppose that we observe $y^\delta\in Y$ such
that $\|y^\delta-y\|\leq \delta.$
We introduce  regularized solutions obtained by a general spectral filter function $g_\alpha$:
\[
x_\alpha=g_\alpha(T^*T)T^*y,\ \ x_\alpha^\delta=g_\alpha(T^*T)T^*y^\delta.
\]
Moreover, let $f\in X^*=X$ be a linear functional.

One aim of this paper is to obtain upper bounds for the error of  linear functionals of the solutions, i.e., 
for the quantity $\lg f, x_{\alpha(y^\delta)}^\delta-x\rg$, where a parameter 
$\alpha(y^\delta)$ is selected in a special way and depends only on the observation $y^\delta.$
To state a smoothness/source condition for $x$ and/or $f$, we use 
$\vf$ and $\vk$, which are  continuous, non-negative, increasing  real functions defined for positive real values
(so-called index functions).
Below we impose some standard assumptions on $\vf,\vk,g_\alpha$.

Convergence rates estimates for the error $x_\alpha-x$ using some smoothness conditions on $x$ 
are nowadays a classical topic. 
For instance, if $\delta$ is known, see, for example, \cite{LuPer13}, then under some natural conditions
the best accuracy that can be guaranteed under the smoothness condition
$x\in \Ran(\vf(T^*T))$ is of the order $ \vf(\theta^{-1}(\delta))$, 
where $\theta(t)=\vf(t)\sqrt{t}$ and $\theta^{-1}$
is its inverse function. For linear functionals, the situation can be improved: 
Assume that $x\in \Ran(\vf(T^*T)),$  and $ f\in \Ran(\vk(T^*T))$, where 
$\vf, \vk$ are index functions, 
then the best accuracy for the linear functionals $\lg f, x_{\alpha(y^\delta)}^\delta-x\rg$
is of the order $(\vk \vf )(\theta^{-1}(\delta))$.

If the noise intensity is known, then the best order in accuracy can usually be achieved by standard 
means of selecting $\alpha$. 
However, if $\delta$ is not known, the choice of the optimal $\alpha$ is a serious problem.
 For  $\alpha(y^\delta)$   selected 
  according to the quasi-optimality principle, some upper  bounds for $\|x_{\alpha(y^\delta)}^\delta-x\|$ 
  were obtained in~\cite{KinNeu08, Neu08}. 
There it is proved that if $\vf(t)=t^\mu$ and if the qualification $\mu_0$ of the regularization 
$g_\alpha$ is such that $\mu_0\geq \mu$, then 
\[
\|x_{\alpha(y^\delta)}^\delta-x\|=O(\delta^{\frac{2\mu}{2\mu+1}\frac{\mu}{\mu_0}}), \ \delta\to0.
\]
The main assumption  on the noise was the following condition of Muckenhoupt type (noise condition):
\be\label{eq:main_assumption}
\exists C>0\ \forall \delta>0\ \forall \alpha>0 \qquad 
\alpha^2\int_\alpha^\infty \lambda^{-1} d\| F_\lambda (y^\delta-y)\|^2 \leq C\int_0^\alpha\lambda   
d\| F_\lambda (y^\delta-y)\|^2.
\ee
%where $\{F_\lambda\}$ is a spectral family of $TT^*, \ Q$ is the orthogonal projector on $\overline{\Ran(T)}.$
%\red{
%We have to either  delete $Q$ here and below or delete the assumption $Ker(T)=\{0\}, Ker(T^*)=\{0\} $
%and write $x$ instead of $x^\dagger $, etc.}

We give some sufficient  conditions that ensure \eqref{eq:main_assumption} in 
Section~\ref{section:sufficient}.
In this paper we  consider \eqref{eq:main_assumption} and its generalization for  
the linear functional strategy. 
We discuss these conditions 
in the deterministic and random case; in particular we verify that for mildly ill-posed problems 
and Gaussian noise, it is satisfied almost surely.  Moreover, we provide 
upper bounds for $\lg f, x_{\alpha(y^\delta)}^\delta-x\rg$, where $\alpha(y^\delta)$ is selected 
by the quasi-optimality principle as in 
\cite{KinNeu08, Neu08}, and  we also obtain some generalization of the upper bounds there.
Furthermore, we prove improved bounds $\lg f, x_{\alp_{\vk}(y^\delta)}^\delta-x\rg$, 
when $\alpha_\vk(y^\delta)$ is selected heuristically but using  information about $y^\delta$
and also  $\vk$.

% The  following questions are discussed.
% \begin{itemize}
% \item Examples  of \eqref{eq:main_assumption1} if the noise is deterministic or  random;
% \item The upper bounds for $\lg f, x_{\alpha(y^\delta)}^\delta-x\rg$, where $\alpha(y^\delta)$ is selected as in 
% \cite{KinNeu08, Neu08}. 
% Here we also obtain some generalization of their  upper bounds.
% \item The upper bounds for $\lg f, x_{\alpha_\vk(y^\delta)}^\delta-x\rg$, 
% where $\alpha_\vk(y^\delta)$ is selected using the information about $y^\delta$
% and also  $\vk$.
% In this case the upper bound for $\lg f, x_{\alpha_\vk(y^\delta)}^\delta-x\rg$ is better than in previous item.
% \end{itemize}

For later use we introduce the quasi-optimality functional and a variant suited for functionals: 
\begin{align*}
\psi^2(\alpha,y^\delta)&=\int_0^\infty (1-\lambda g_\alpha(\lambda))^2 \lambda g_\alpha^2(\lambda) 
d \|F_\lambda y^\delta\|^2=
\|(I-T^*Tg_\alpha(T^*T))x_\alpha^\delta\|^2, \\
\psi_\vk^2(\alpha,y^\delta)&=\int_0^\infty \vk^2(\lambda)(1-\lambda g_\alpha(\lambda))^2 \lambda
g_\alpha^2(\lambda) d \|F_\lambda y^\delta\|^2\\
&=\|\vk(T^*T)(I-T^*Tg_\alpha(T^*T))x_\alpha^\delta\|^2.
\end{align*}
We introduce the following minimization-based heuristic parameter choice rules; the first one 
is the classical quasi-optimality rule as in \cite{KinNeu08, Neu08} while the second one 
is our modification:
\begin{equation} \alpha(y^\delta)=\textrm{argmin}_\alpha\psi(\alpha,y^\delta), \qquad 
\alpha_\vk(y^\delta)=\textrm{argmin}_\alpha\psi_\vk(\alpha,y^\delta).
\end{equation}

It is clear that $\alpha(y^\delta)$ can be computed without knowledge of $\delta$, which is 
the defining feature of heuristic parameter choice rules. The novel modified rule 
$\alpha_\vk(y^\delta)$ additionally needs knowledge of the functional smoothness (via $\vk)$). 
It will be shown that this additional information leads to improvements in the error bounds. 

%\red{???? Should some argumentation  be added? Why minimum is attained???}
%\begin{remk}
%In the paper we will use  a lot of constants $K, C, c_1, c_2,...$; the same notation  may be used for different constants.
%end{remk}
%\vskip 10pt
At first, we state some standard assumptions:
\begin{assumption}\label{assone} \ 

\begin{enumerate}
\item For all $\alpha>0$ we have
\be\label{eq:assumption_g}
0\leq \lambda g_\alpha(\lambda) \leq 1,\lambda>0, \qquad \sup_{\lambda>0}  
\sqrt{\lambda}  g_\alpha(\lambda)\leq \frac{c_1}{\sqrt{\alpha }}.
\ee
\item For all $  \alpha>0$
and $\lambda\in(0,\alpha)$
\be\label{eq:assumption_g_2}
 (1-\lambda g_\alpha(\lambda))\geq c_2, \qquad  \frac{c_3}{ \alpha}\leq g_\alpha(\lambda) \leq \frac{c_4}{\alpha}.
  \ee

  \item For any $\lambda>0$, 
\be\label{eq:assumption_g_3}
k(\lambda):=\inf_{\alpha\in(0,\|T\|]}\frac{(1-\lambda g_\alpha(\lambda))g_\alpha(\lambda)}{\alpha} >0. 
\ee
%
%\red{
%
%}
%   
% \item $ \forall \alpha_0>0\ \exists c=c(\alpha_0)\ \exists \alpha_1>0\ \forall \alpha\in(0,\alpha_1) \ \forall \lambda \geq \alpha_0  $
% \be\label{eq:assumption_g_3}
%  (1-\lambda g_\alpha(\lambda))\geq c \alpha,  \ \ \mbox{and} \ \    \inf_{\lambda\geq \alpha_0}g_\alpha(\lambda) >0;
% \ee
\item The qualification of $g_\alpha$ covers $\vf  $ and $\vf\vk$, i.e., for all $\alpha>0$
\be\label{eq:assumption_qualification}
 \sup_{\lambda>0} |\vf(\lambda)(1-\lambda g_\alpha(\lambda))|\leq c_5  \vf(\alpha),
\ee
\be\label{eq:assumption_qualification_2} 
 \sup_{\lambda>0} |\vk(\lambda)\vf(\lambda)(1-\lambda g_\alpha(\lambda))|\leq c_6 \vk(\alpha)\vf(\alpha).
\ee
\item The function $\vk$ is covered by the qualification $1/2$, i.e.,  for all $\alpha>0$
\be\label{eq:299}
\sup_{\lambda>\alpha}  \vk(\lambda)/\sqrt{\lambda} \leq c_7\vk(\alpha)/\sqrt{\alpha}.
\ee
\item  The function $\vk,\vf$ are regularly varying: For all  $c_8>0$  there exists $c_9>0$ and $\delta_0>0$ such that 
\be\label{regvar} \vf(c_8 \delta) \leq c_9\vf(\delta)  \quad \mbox{ and  } \quad  
\vk(c_8 \delta) \leq c_9\vk(\delta)  \qquad  \forall \delta\in (0,\delta_0). 
\ee
\end{enumerate}
\end{assumption}
% \item 
% \be\label{eq:assumption_ineq_psi}
%   \exists \delta_0>0\ \forall \delta\in(0,\delta_0):\ \ \ \ \psi(\alpha,y^\delta)\geq c_8 \alpha.
% \ee

We note that in several places, condition \eqref{eq:assumption_g_3} could be replaced by 
one with a more general qualification, i.e., that 
there exists $\mu_0>0$   such that for   any $\lambda>0$ 
\be\label{eq:assumption_g_3_1}
k(\lambda):=\inf_{\alpha\in(0,\|T\|]}\frac{(1-\lambda g_\alpha(\lambda))g_\alpha(\lambda)}{\alpha^{\mu_0}} >0.
\ee

Additionally to the structural conditions on the filter and index functions, 
we impose the following  generalization of the noise condition \eqref{eq:main_assumption}:
% there exists   $\delta_0>0 $ such that  for all $\delta\in (0, \delta_0)$ and $\alpha>0$
\begin{equation}\label{eq:main_assumption_kappa}
\begin{split}
&\exists \delta_0>0:  \forall \delta\in (0, \delta_0) \forall \alpha>0: \\
& \alpha^2\int_\alpha^\infty \lambda^{-1} \vk^2(\lambda)d\| F_\lambda (y^\delta-y)\|^2 \leq 
c_{10}\int_0^\alpha\lambda   \vk^2(\lambda)d\| F_\lambda (y^\delta-y)\|^2.
\end{split} 
\end{equation}

We state the main convergence result of the paper. In the sequel we denote by 
$\vee$ the maximum. 
\begin{thm}\label{thm:main}
Suppose that $y\neq 0, x\in \Ran(\vf(T^*T)), f\in \Ran(\vk(T^*T)),$  where $\vf, \vk$ are continuous, 
non-negative, increasing  functions, the function 
$(0,\infty)^2\ni(\lambda,\alpha)\to g_\alpha(\lambda)$ is continuous, 
and there are constants $c_1,\ldots,c_9>0$ such that Assumptions~\ref{assone} hold. 
Moreover, let the noise condition \eqref{eq:main_assumption_kappa} hold. 

Then, as $\delta\to0$,
\begin{align}
 \label{eq_main_estimate0}
 |x_{\alpha(y^\delta)}^\delta-x| &=O\left( \vf(\vf(\theta^{-1}(\delta)) \vee  \vf(\theta^{-1} (\delta ))\right), \\
% \ee
% \[
\begin{split}
  |\lg f, x_{\alpha(y^\delta)}^\delta-x\rg| &=
  O\left(\vk(\vf(\theta^{-1}(\delta)))\vf(\vf(\theta^{-1}(\delta)) \vee  \vf(\theta^{-1} (\delta ))\right)\\&=
% \]
% \be
 O\left( (\vk \vf)\circ(\vf(\theta^{-1}(\delta)))\vee  \vf(\theta^{-1} (\delta ))\right);
 \end{split} \label{eq_main_estimate1} \\
% \ee
% \be
\label{eq_main_estimate2}
  |\lg f, x_{\alpha_\vk(y^\delta)}^\delta-x\rg| &=O\left( (\vk \vf)\circ(\vk\vf)(\theta^{-1}(\delta)) 
  \vee  (\vk\vf)(\theta^{-1} (\delta ))\right).
\end{align}% \delta\to0.
% \ee
%
%
%
\end{thm}
Observe, that the bound \eqref{eq_main_estimate2} for the modified rule $\alpha_\vk(y^\delta)$
is improved compared to \eqref{eq_main_estimate1}.

\begin{remk}
If we replace \eqref{eq:assumption_g_3} by the more general one, \eqref{eq:assumption_g_3_1}, 
then the convergence rates in this theorem read as 
\begin{align}
|x_{\alpha(y^\delta)}^\delta-x| &=O\left( \vf(\vf^{1/\mu_0}(\theta^{-1}(\delta)) \vee  \vf(\theta^{-1}
(\delta ))\right),\label{eq_main_estimate0_1} \\
\begin{split} |\lg f, x_{\alpha(y^\delta)}^\delta-x\rg| &=O\left(\vk(\vf^{1/\mu_0}(\theta^{-1}
 (\delta)))\vf(\vf(\theta^{-1}(\delta)) \vee  \vf(\theta^{-1} (\delta ))\right)
\\
%\be
&=O\left( (\vk \vf)\circ(\vf^{1/\mu_0}(\theta^{-1}(\delta)))\vee  \vf(\theta^{-1} (\delta ))\right),\label{eq_main_estimate1_1}
\end{split}
\\
%\be
 |\lg f, x_{\alpha_\vk(y^\delta)}^\delta-x\rg| &=
 O\left( (\vk \vf)\circ(\vk\vf)^{1/\mu_0}(\theta^{-1}(\delta)) \vee  (\vk\vf)(\theta^{-1} 
 (\delta ))\right).\label{eq_main_estimate2_1}
\end{align}
\end{remk}

\begin{remk}
Formula \eqref{eq_main_estimate0} can be deduced using the reasoning of  \cite{KinNeu08, Neu08}
(the authors used concrete  power function in their estimates). It also can be seen from our proof for $\vk(\lambda)\equiv 1.$
To verify  \eqref{eq_main_estimate0}, actually  only \eqref{eq:main_assumption} is required, which 
is implied by  \eqref{eq:main_assumption_kappa} as the following remark indicates.
\end{remk}
%\begin{remk}
%\red{May be we should write some comments about a choice of qualification $\mu_0$}
%\end{remk}
\begin{remk}
The main assumption of the Theorem is \eqref{eq:main_assumption_kappa}. It can be considered as 
an analogue of \eqref{eq:main_assumption} from \cite{KinNeu08, Neu08} 
for the  mollified noise $\vk(T^*T)(y^\delta-y)$.
It should be noted, that \eqref{eq:main_assumption_kappa} implies \eqref{eq:main_assumption}. Indeed, 
it follows from the monotonicity of $\vk$
that $\frac{\vk(\lambda)}{\vk(\alpha)}\geq 1$ for $ \lambda\geq \alpha$. So,
\[
\alpha^2\int_\alpha^\infty \lambda^{-1} d\| F_\lambda (y^\delta-y)\|^2 \leq %\mbox {monotonicity } 
\alpha^2\int_\alpha^\infty \lambda^{-1}\frac{\vk^2(\lambda)}{\vk^2(\alpha)} d\| F_\lambda (y^\delta-y)\|^2.
\]
Due to \eqref{eq:main_assumption_kappa} the right hand side of the last inequality is less than or equal to
 \[
c_{10} \int_0^\alpha  \lambda \frac{\vk^2(\lambda)}{\vk^2(\alpha)} d\| F_\lambda (y^\delta-y)\|^2 \leq
 c_{10}\int_0^\alpha\lambda   d\| F_\lambda (y^\delta-y)\|^2,
\]
where we used that $\frac{\vk(\lambda)}{\vk(\alpha)}\leq 1$ for $\lambda\leq \alpha$.
\end{remk}

  For Tikhonov's regularization  $g_\alpha(\lambda)=\frac{1}{\alpha+\lambda},$  assumptions 
  \eqref{eq:assumption_g}, \eqref{eq:assumption_g_2}, and
\eqref{eq:assumption_g_3} are obviously satisfied, 
% assumption    \eqref{eq:assumption_ineq_psi} was proved in  \cite{KinNeu08, Neu08},
 assumptions 
\eqref{eq:assumption_qualification}, \eqref{eq:assumption_qualification_2}, and  \eqref{eq:299}
 are valid for $\vf(t)=t^\mu,$ $\vk(t)=t^\gamma$ with $\mu>0, \gamma\in[0,1/2],  \mu+\gamma\leq 1.$ 
% \red{

For iterated  Tikhonov's regularization  
$g_\alpha(\lambda)=\lambda^{-1}(1-\frac{\alpha^n}{(\alpha+\lambda)^n}),$  
assumptions \eqref{eq:assumption_g}, \eqref{eq:assumption_g_2}, and
\eqref{eq:assumption_g_3_1} are obviously satisfied, where $\mu_0=n$; 
% assumption    \eqref{eq:assumption_ineq_psi} was proved in  \cite{KinNeu08, Neu08},
 assumptions 
\eqref{eq:assumption_qualification}, \eqref{eq:assumption_qualification_2}, and  \eqref{eq:299}
 are valid for $\vf(t)=t^\mu,$ \mbox{$\vk(t)=t^\gamma$} with $\mu>0, \gamma\in[0,1/2],  \mu+\gamma\leq \mu_0.$
%}
 
Specializing the previous theorem to Tikhonov regularization and H\"older-type index functions, we find the 
following corollary:
\begin{corl}\label{corl:1}
Let  $g_\alpha(\lambda)=\frac{1}{\alpha+\lambda},$ $\vf(t)=t^\mu, \vk(t)=t^\gamma$ with
$\mu>0, \gamma\in[0,1/2], \mu+\gamma\leq 1.$
Assume that \eqref{eq:main_assumption_kappa} is satisfied. Then as $\delta \to 0$, 
\begin{align*}
 |  x_{\alpha(y^\delta)}^\delta-x| &= 
O\left( \delta^{\frac{2\mu}{2\mu+1}  \mu} \right), \\
 |\lg f, x_{\alpha(y^\delta)}^\delta-x\rg| &= 
O\left( \delta^{\frac{2\mu}{2\mu+1} (\mu+\gamma)}\right),\\
 |\lg f, x_{\alpha_\vk(y^\delta)}^\delta-x\rg| &=O\left( 
\delta^{\frac{2(\mu+\gamma)^2}{2\mu+1}} \right). 
\end{align*}
\end{corl}

\begin{remk}
If we use the generalized qualification condition \eqref{eq:assumption_g_3_1}
and replace the condition $\mu+\gamma \leq 1$ by $\mu+\gamma\leq \mu_0$,
then the rates in Corollary~\ref{corl:1} have to be replaced by  
% 
% 
% \red{$\mu+\gamma\leq \mu_0$.}
% \red{$\mu+\gamma\leq \mu_0$.}
% 
% \red{In case of a general qualification $\mu_0$ we have 
\begin{align*}
 |  x_{\alpha(y^\delta)}^\delta-x| &= 
O\left( \delta^{\frac{2\mu}{2\mu+1}  \frac{\mu}{\mu_0}} \right), \qquad 
 |\lg f, x_{\alpha(y^\delta)}^\delta-x\rg| = 
O\left( \delta^{\frac{2\mu}{2\mu+1} \frac{\mu+\gamma}{\mu_0}}\right), \\
 |\lg f, x_{\alpha_\vk(y^\delta)}^\delta-x\rg| &=O\left( 
\delta^{\frac{2(\mu+\gamma) }{2\mu+1 }\frac{ \mu+\gamma  }{ \mu_0}} \right).
\end{align*}
\end{remk}

\begin{remk}
Under the conditions of Corollary \ref{corl:1}  the bound  for $\|x_{\alpha(y^\delta)}^\delta-x\|$ in  \cite{KinNeu08, Neu08} is
$O\left( \delta^{\frac{2\mu}{2\mu+1}  \mu} \right)$ 
(respectively, 
$O\left( \delta^{\frac{2\mu}{2\mu+1}  \frac{\mu}{\mu_0}} \right)$ for the case with $\mu_0$)
while  the order-optimal bound is $O\left( \delta^{\frac{2\mu}{2\mu+1}} \right)$.
For linear functionals as in the corollary, it is known that the optimal order is 
$|\lg f, x_{\alpha}^\delta-x\rg|=O\left( 
\delta^{\frac{2(\mu+\gamma)}{2\mu+1}} \right),$ as $\delta\to 0;$ see \cite{LuPer13}.
\end{remk}

\section{Proof of the main result}\label{section:proof_main}
We need the following auxiliary results. Many of them are quite standard, 
we provide the proofs to make the exposition self-contained.
At first we provide bounds for the approximation errors. 
\begin{lem}\label{lem:f_estim}
Under Assumption~\ref{assone}, there is $c>0 $ such that for all $\alpha>0$ we have
\begin{align*}
\|x_\alpha-x\| &\leq c \vf(\alpha); \\
|\lg f, x_\alpha-x\rg| &\leq c \vk(\alpha)\vf(\alpha);\\
\|\vk(T^*T)(x_\alpha-x)\| &\leq c \vk(\alpha)\vf(\alpha).
\end{align*}
\end{lem}
\begin{proof}
Let $x=\vf(T^*T)v_x, f=\vk(T^*T)u_f.$ Then
\begin{align*}
\| x_\alpha-x\|^2 &=
 \int_0^\infty  (1-\lambda g_\alpha(\lambda))^2 d \|E_\lambda x\|^2=
 \int_0^\infty  \vf^2(\lambda)(1-\lambda g_\alpha(\lambda))^2 d \|E_\lambda v_x\|^2\\
&\leq  K_1 \sup_\lambda (\vf(\lambda)(1-\lambda g_\alpha(\lambda)))^2 \leq K_2 
   \vf^2(\alpha),
\end{align*}  
which proves the first inequality. For the remain ones, we estimate
\begin{align*}
&\lg f, x_\alpha-x\rg^2 =
\lg \vk(T^*T)u_f, x_\alpha-x\rg^2=  \lg u_f, \vk(T^*T)(x_\alpha-x)\rg^2\\
& \qquad  \leq
\|u_f\|^2 \|\vk(T^*T)(x_\alpha-x)\|^2=
\|u_f\|^2 \int_0^\infty \vk^2(\lambda)(1-\lambda g_\alpha(\lambda))^2 d \|E_\lambda x\|^2\\
& \qquad  =
\|u_f\|^2 \int_0^\infty \vk^2(\lambda)\vf^2(\lambda)(1-\lambda g_\alpha(\lambda))^2 d \|E_\lambda v_x\|^2\\
&\qquad \leq
K_1 \sup_\lambda (\vk(\lambda)\vf(\lambda)(1-\lambda g_\alpha(\lambda)))^2 \leq K_2 
   \vk^2(\alpha)\vf^2(\alpha),
\end{align*} 
where we used \eqref{eq:assumption_qualification_2}.
\end{proof}

Next we bound the parameter choice functionals.  
\begin{lem}\label{lem:psi_estim}
Let Assumption~\ref{assone} hold.  
Then there exists a  $c>0$ such that for all  $\alpha>0$ and all $\delta>0$
   we have
\begin{align*} 
\psi(\alpha,y)&\leq \|x_\alpha-x\|, \qquad \qquad \qquad 
\psi(\alpha,y^\delta-y)\leq \|x^\delta_\alpha-x_\alpha\|, \\
\psi(\alpha,y^\delta)&\leq \|x_\alpha-x\|+\|x^\delta_\alpha-x_\alpha\|, \\
\psi_\vk(\alpha,y)&\leq \|\vk(T^*T)(x_\alpha-x)\|,    \qquad  \psi_\vk(\alpha,y^\delta-y)\leq 
\|\vk(T^*T)(x^\delta_\alpha-x_\alpha)\|,\\
\psi_\vk(\alpha,y^\delta)&\leq \|\vk(T^*T)(x_\alpha-x)\|+\|\vk(T^*T)(x^\delta_\alpha-x_\alpha)\|.
\end{align*}
\end{lem}
\begin{proof}
\begin{align*}
&\psi^2(\alpha,y)=\| (I-T^*Tg_\alpha(T^*T))x_\alpha\|^2=
\| (I-T^*Tg_\alpha(T^*T)) T^*T g_\alpha(T^*T)x \|^2\\
&\quad= 
\int^\infty_0 (1-\lambda g_\alpha^2(\lambda))^2 (\lambda g_\alpha(\lambda))^2 d\|E_\lambda x\|^2
\leq \int^\infty_0 (1-\lambda g_\alpha^2(\lambda))^2   d\|E_\lambda x\|^2 = \|x_\alpha-x\|^2. \\
%\end{align*}
%\begin{align*}
&\psi(\alpha,y^\delta-y)= 
\|(I-T^*Tg_\alpha(T^*T))(x_\alpha^\delta-x_\alpha)\|\leq \|x^\delta_\alpha-x_\alpha\|.
\end{align*}
The inequalities for $\psi_\vk$ follow in an analogous way. 
\end{proof}
The following result is a straightforward consequence of \eqref{eq:assumption_g} and $\|y^\delta-y\|\leq \delta$.
\begin{lem}\label{lem:x_delta_x}
Let Assumption~\ref{assone} hold. 
There exists $c>0 $ such that for all  $\alpha>0$ and all $\delta>0$
   we have
\[
\|x_\alpha^\delta-x_\alpha\|^2=\int_0^\infty \lambda g^2_\alpha(\lambda) d\|F_\lambda (y^\delta-y)\|^2 \leq c \left( \frac{\delta}{\sqrt{\alpha}}\right)^2.
\]
\end{lem}
\begin{lem}\label{lem:estim_psi}
Let Assumption~\ref{assone} hold. 
We have for $\delta>0,$ 
\begin{align} \label{eq:378}
\psi(\alpha(y^\delta), y^\delta)&= \inf_\alpha    \psi(\alpha,y^\delta)\leq \inf_\alpha \left(\|x_\alpha-x\|+\|x^\delta_\alpha-x_\alpha\|\right)
\leq c_0 \vf(\theta^{-1}(\delta)),  \\
\psi_\vk(\alpha(y^\delta), y^\delta)&\leq c_1 \psi(\alpha(y^\delta), y^\delta) \leq c_2 \vf(\theta^{-1}(\delta)), \ 
\label{eq:psi_kappa}
\end{align} 
where $c_0, c_1, c_2$ are constants  independent of $\delta,$  $\theta(t)=\vf(t)\sqrt{t}$, 
and $\theta^{-1}$ is its inverse function.
% \be\label{eq:psi_kappa_2}
% \psi_\vk(\alpha, y)\leq c_2 \vk(\alpha)\vf(\alpha), \ \delta>0.
% \ee
\end{lem}
\begin{proof}
Let $\bar \alpha $ be such that $\vf(\bar \alpha)=\frac{\delta}{\sqrt{\bar \alpha}},$ i.e., 
${\bar \alpha}=\theta^{-1}(\delta).$ Then \eqref{eq:378}
  follows from Lemmas \ref{lem:f_estim} and \ref{lem:x_delta_x}, and the following calculations
\begin{align*} 
\inf_\alpha \left(\|x_\alpha-x\|+\|x^\delta_\alpha-x_\alpha\|\right)&\leq 
C\inf_\alpha \left(\vf(\alpha)+\frac{\delta}{\sqrt{\alpha}}\right)\leq 
C  \left(\vf(\bar \alpha)+\frac{\delta}{\sqrt{\bar \alpha}}\right) \\
&=
2C \vf(\theta^{-1}(\delta)).
\end{align*}
Inequality \eqref{eq:psi_kappa} follows from \eqref{eq:378} because $\vk$ is bounded on $[0,\|T\|]$.
\end{proof}
The next lemma gives a very important consequence of \eqref{eq:main_assumption_kappa}, which is crucial 
for our proofs. In the sequel, we use the
symbols $K_1,K_2,\ldots$, and $C$ for generic constants that may take different values in different formulas.
\begin{lem}\label{lem:nontrivial_estimate}
Let Assumption~\ref{assone} hold and assume the generalized noise condition 
\eqref{eq:main_assumption_kappa}.  Then 
there exist constants $K_1, K_2,$ and $\delta_0>0$ such that for all $\delta\in(0,\delta_0),\ \alpha>0$:
\[
|\lg f,  (x_\alpha^\delta-x_\alpha)\rg| \leq K_1 \|\vk(T^*T)(x_\alpha^\delta-x_\alpha)\| \leq K_2 \psi_\vk(\alpha, y^\delta-y).
\]
\end{lem}
\begin{proof} The first inequality is proved similarly to Lemma \ref{lem:f_estim}. 
Let us verify the second inequality. By splitting the integral we obtain
\[
\|\vk(T^*T)(x_\alpha^\delta-x_\alpha)\|^2 =
\int_0^\infty \vk^2(\lambda)\lambda g_\alpha^2(\lambda) d \|F_\lambda (y^\delta- y)\|^2=
\int_0^\alpha[\dots]+\int_\alpha^\infty[\dots].
\]
It follows from \eqref{eq:assumption_g} and \eqref{eq:main_assumption_kappa}  that
\begin{align*}
&\int_\alpha^\infty \vk^2(\lambda)\lambda g_\alpha^2(\lambda) d \|F_\lambda (y^\delta- y)\|^2=
\int_\alpha^\infty \vk^2(\lambda)\lambda^{-1} (\lambda g_\alpha(\lambda))^2 d \|F_\lambda (y^\delta- y)\|^2\\
& \qquad \leq
\int_\alpha^\infty \vk^2(\lambda)\lambda^{-1}  d \|F_\lambda (y^\delta- y)\|^2 \leq 
K_1 \alpha^{-2} \int_0^\alpha \lambda   \vk^2(\lambda)d\| F_\lambda (y^\delta-y)\|^2.
\end{align*}
The second assumption in \eqref{eq:assumption_g_2} yields 
that $g^2_\alpha(\lambda)\leq \frac{\const}{\alpha^2}$ for $ \lambda\in(0,\alpha).$ Thus, 
\[
\int_0^\alpha \vk^2(\lambda)\lambda g_\alpha^2(\lambda) d \|F_\lambda (y^\delta- y)\|^2
\leq K_2 \alpha^{-2} \int_0^\alpha \vk^2(\lambda)\lambda   d \|F_\lambda (y^\delta- y)\|^2,
\]
and consequently
\[
\|\vk(T^*T)(x_\alpha^\delta-x_\alpha)\|^2 \leq K_3 \alpha^{-2} \int_0^\alpha \vk^2(\lambda)\lambda   d \|F_\lambda (y^\delta- y)\|^2.
\]
Since $(1-\lambda g_\alpha(\lambda))\geq \const>0$ and $g^2_\alpha(\lambda)\geq \frac{\const}{\alpha^2}>0$ 
for $ \lambda\in(0,\alpha), $ see \eqref{eq:assumption_g_2}, we have
\begin{align*}
&\alpha^{-2}\int_0^\alpha \vk^2(\lambda)\lambda   d \|F_\lambda (y^\delta- y)\|^2\leq K_4 
  \int_0^\alpha \vk^2(\lambda)\lambda g^2_\alpha(\lambda) (1-\lambda g_\alpha(\lambda))^2  d
  \|F_\lambda (y^\delta- y)\|^2
\\
& \qquad \leq K_4 \int_0^\infty \vk^2(\lambda)\lambda g^2_\alpha(\lambda)(1-\lambda g_\alpha(\lambda))^2  d \|F_\lambda (y^\delta- y)\|^2= K_4 \psi^2_\varkappa(\alpha, y^\delta-y).
\end{align*}
\end{proof}

% Here we are

\begin{lem}\label{lem:bound_y_kappa} 
Let $y\not = 0$. Then there exist $C>0 $ and $\delta_0>0$ 
 such that for all $\delta\in(0,\delta_0) \mbox{ and } \alpha\in(0,1)$
\be %\label{eq:assumption_ineq_psi}
  \psi(\alpha,y^\delta)\geq C \alpha \qquad \mbox{ and } \qquad 
%\ee
%and 
%\be
   \psi_\vk(\alpha,y^\delta)\geq C \alpha. \label{eq:assumption_ineq_psi_kappa}
\ee
\end{lem}
\begin{proof}
Let us only verify the second inequality. We follow the course of the proof from  \cite{KinNeu08, Neu08}. 
Let $\bar \alpha$ be fixed. It follows from \eqref{eq:assumption_g_2} that
\begin{align*}
\psi_\vk^2(\alpha,y^\delta) &
=\int_0^\infty \vk^2(\lambda)(1-\lambda g_\alpha(\lambda))^2 \lambda g_\alpha^2(\lambda) 
d \|F_\lambda y^\delta\|^2\\
&\geq
 \int_{\bar \alpha}^\infty \vk^2(\lambda) \alpha^2(\frac{(1-\lambda g_\alpha(\lambda))
 g_\alpha(\lambda)}{\alpha})^2 \lambda   d \|F_\lambda y^\delta\|^2\\&
 \geq
 \int_{\bar \alpha}^\infty \vk^2(\lambda) \alpha^2\inf_{a\in(0,\|T\|]}\left(\frac{(1-\lambda g_a(\lambda))
 g_a(\lambda)}{a}\right)^2 \lambda   d \|F_\lambda y^\delta\|^2\\
 &=
 \alpha^2\int_{\bar \alpha}^\infty  \vk^2(\lambda)  k^2(\lambda)\lambda d \|F_\lambda y^\delta\|^2\\
 &\geq
  \alpha^2 \left( \int_{\bar \alpha}^\infty \vk^2(\lambda) k^2(\lambda)\lambda d\left[ 2^{-1}
  \|F_\lambda y\|^2 -\|F_\lambda (y^\delta-y)\|^2\right]\right),
\end{align*}
% \[
%  \int_{\bar \alpha}^\infty K_7 \alpha^2 \lambda g_\alpha^2(\lambda) d \|F_\lambda y^\delta\|^2\geq
% \]
 where $k$ is from \eqref{eq:assumption_g_2}. Set $h(\lambda)=
 \vk^2(\lambda)k^2(\lambda)\lambda, \lambda>0$;  the function $h$ is positive.
It follows from the definition of  $k$ that   $h(\lambda)\leq \vk^2(\lambda) 
(1-\lambda g^2_1(\lambda))^2\lambda g^2_1(\lambda)$. So, all considered integrals are finite.
% 
% 
% \[
% K_7 \alpha^2 \int_{\bar \alpha}^\infty    d \|F_\lambda y^\delta\|^2 \geq 
% \]
%\[
% K_7 \alpha^2  \left( 2^{-1}\int_{\bar \alpha}^\infty    d \|F_\lambda y\|^2 -\alpha^2 \int_{\bar \alpha}^\infty    d \|F_\lambda (y^\delta-y)\|^2\right).
%\]
Select $ \bar \alpha>0$ such that $\int_{\bar \alpha}^\infty    d \|F_\lambda y\|^2>0.$ Then
$\int_{\bar \alpha}^\infty   h(\lambda) d2^{-1} \|F_\lambda y\|^2>0.$ 
Since 
\[
\lim_{\delta\to0}\int_{\bar \alpha}^\infty    h(\lambda) d\|F_\lambda (y^\delta-y)\|^2=0,
\]
there is $\delta_0>0$ such that 
\[
\int_{\bar \alpha}^\infty   h(\lambda) d\|F_\lambda (y^\delta-y)\|^2\leq 4^{-1}
\int_{\bar \alpha}^\infty   h(\lambda) d  \|F_\lambda y\|^2.
\]
Hence  
we get the second inequality in \eqref{eq:assumption_ineq_psi_kappa} with %   $\alpha_0 =\bar \alpha$ and 
$C=4^{-1}
\int_{\bar \alpha}^\infty  h(\lambda) d \|F_\lambda y\|^2$.
%Inequality  can be proved similarly because $\vk(\lambda)>0, \lambda>0$ and $\vk$ is bounded on $[0,\|T\|]$.
\end{proof}

Bounds for $\psi_\vk(\alpha_\vk(y^\delta), y^\delta) $ are given in the following statement. 
\begin{lem}\label{lem:estim_psi_kappa}
We have
\be\label{eq:539}
\psi_\vk(\alpha_\vk(y^\delta), y^\delta)
\leq C \vk(\theta^{-1}(\delta))\vf(\theta^{-1}(\delta)), 
\ee
   where $C$ is a  constant   independent of $\delta.$ 
\end{lem}
\begin{proof} Lemmas   \ref{lem:f_estim} and \ref{lem:psi_estim} yield  that
\[
\psi_\vk(\alpha_\vk(y^\delta), y^\delta)= 
\inf_\alpha     \psi_\vk(\alpha,y^\delta)\leq 
\inf_\alpha \left(   \|\vk(T^*T)(x_\alpha-x)\|+\|\vk(T^*T)(x^\delta_\alpha-x_\alpha)\|\right).
\]
It follows from \cite[page 107]{LuPer13} that the right hand side is less than or equal to
$C \vk(\theta^{-1}(\delta))\vf(\theta^{-1}(\delta)).$
% \[
% C_1 \inf_\alpha \left(   (\int_0^\infty \vk^2(\lambda)\vf^2(\lambda)(1-\lambda g_\alpha(\lambda))^2   d \|E_\lambda v_x\|^2)^{1/2} + (\int_0^\infty \vk^2(\lambda)(1-\lambda g_\alpha(\lambda))^2 \lambda g_\alpha^2(\lambda) d \|F_\lambda (y^\delta-y)\|^2)^{1/2}\right)\leq
% \]
% \[
% C_2 \inf_\alpha \sup_\lambda\left( \vk(\lambda)\vf(\lambda)  + 
% \delta  \vk (\lambda)(1-\lambda g_\alpha(\lambda))  (\lambda^{1/2} g_\alpha (\lambda)) \right)\leq
% \]

% \[
% C_3 \inf_\alpha \sup_\lambda\left[ \vk(\lambda)\left(\vf(\lambda)  + \frac{\delta   }{\sqrt{\alpha}}
%  \right)\right]\leq
% \]
\end{proof}

\begin{proof}[Proof of Theorem  \ref{thm:main}]
By Lemmas~\ref{lem:f_estim}, \ref{lem:nontrivial_estimate}, and \ref{lem:psi_estim},
\allowdisplaybreaks
\begin{align}
& |\lg f, x_{\alpha(y^\delta)}^\delta-x\rg|
 \leq  |\lg f, x_{\alpha(y^\delta)}^\delta-x_{\alpha(y^\delta)}\rg|+ 
 |\lg f, x_{\alpha(y^\delta)}-x\rg|\\
 & \stackrel{\leq}{\mbox{\scriptsize Lemma~\ref{lem:f_estim}}} \
    |\lg f, x_{\alpha(y^\delta)}^\delta-x_{\alpha(y^\delta)}\rg|+  
    K_1 \vk(\alpha(y^\delta))\vf(\alpha(y^\delta))\\
    &\stackrel{\leq}{\mbox{\scriptsize Lemma \ref{lem:nontrivial_estimate}}}\ 
K_2 \left(\psi_\vk(\alpha(y^\delta), y^\delta-y) + \vk(\alpha(y^\delta))\vf(\alpha(y^\delta))\right)\\
%\end{align}
 &\leq 
 K_2 \left(\psi_\vk(\alpha(y^\delta), y^\delta) +\psi_\vk(\alpha(y^\delta), y) + 
 \vk(\alpha(y^\delta))\vf(\alpha(y^\delta))\right)\\
 &\stackrel{\leq}{\scriptsize 
 \mbox{Lemmas \ref{lem:f_estim}, \ref{lem:psi_estim} }}
 K_2 \left(\psi_\vk(\alpha(y^\delta), y^\delta) + K_3 \vk(\alpha(y^\delta))\vf(\alpha(y^\delta)) +
 \vk(\alpha(y^\delta))\vf(\alpha(y^\delta))\right)\\
& \leq  
 %\end{align}
% \be
K_4 \left(\psi_\vk(\alpha(y^\delta), y^\delta) + 
 \vk(\alpha(y^\delta))\vf(\alpha(y^\delta))\right)
  \label{eq:514}\\
 &
 \stackrel{\leq}{ \mbox{ \scriptsize see \eqref{eq:psi_kappa} }}
 K_5 \left(\vf(\theta^{-1}(\delta)) +  \vk(\alpha(y^\delta))\vf(\alpha(y^\delta))\right). \label{eq:473}
\end{align}
% Notice  that $\lim_{\delta\to0}\alpha(y^\delta)=0.$
% Indeed, otherwise there exists a sequence $\{\delta_n\}, \ \lim_{n\to\infty}\delta_n=0$ such that
% $\lim_{n\to\infty}\alpha(y^{\delta_n})=C\in (0,\infty)$. By the Lebesgue dominated convergence theorem we get
% $\lim_{n\to\infty}\psi(\alpha(y^{\delta_n}),  y^{\delta_n} )=
%  \psi(C, y)$. Therefore 
% \[
% 0=\lim_{n\to\infty}\inf_\alpha\psi(\alpha,  y^{\delta_n} )=\lim_{n\to\infty}\psi(\alpha(y^{\delta_n}),  y^{\delta_n} )=
% \]
% \[
%  \psi(C, y)=
% \sqrt{ \int_0^\infty (1-\lambda g_C(\lambda))^2 \lambda g_C^2(\lambda) d \|F_\lambda y\|^2}. 
% \]
% The last expression is positive by our assumptions. This contradicts the assumption $\lim_{\delta\to0}\alpha(y^\delta)\neq 0.$

It follows from Lemma \ref{lem:bound_y_kappa} and \eqref{eq:378} that
\[
\alpha(y^\delta)\leq K_8 \psi(\alpha(y^\delta), y^\delta)\leq K_9
\vf(\theta^{-1}(\delta))
\]
for sufficiently small $\delta>0.$
Thus, the monotonicity of $\vk$ and $\vf $ and \eqref{regvar} 
yields that the right hand side of \eqref{eq:473} does not exceed
\[
K_{10} \left(\vf(\theta^{-1}(\delta)) + \vk(\vf(\theta^{-1}(\delta)))\vf(\vf(\theta^{-1}(\delta)))\right).
\]
This proves \eqref{eq_main_estimate1}. 

%Let us verify \eqref{eq_main_estimate2}. Basically,
The proof of \eqref{eq_main_estimate2} is 
identical to that  of \eqref{eq_main_estimate1}.
Similarly to \eqref{eq:514} we get
\be\label{eq:545}
 |\lg f, x_{\alpha_\vk(y^\delta)}^\delta-x\rg|\leq  K_{11} \left(\psi_\vk(\alpha_\vk(y^\delta), y^\delta) +  \vk(\alpha_\vk(y^\delta))\vf(\alpha_\vk(y^\delta))\right).
\ee
%Similarly to the proof above we get $\lim_{\delta\to0}\alpha_\vk(y^\delta)=0.$

It follows from Lemma \ref{lem:bound_y_kappa} that 
$\alpha_\vk(y^\delta)\leq \psi_\vk(\alpha_\vk(y^\delta), y^\delta). $ The proof of the 
Theorem~\ref{thm:main} now follows from \eqref{eq:545}, \eqref{eq:539}. 
\end{proof}

\section{Case studies of noise conditions}\label{section:sufficient}
%
%Examples of noises that ensure \eqref{eq:main_assumption} and 
%\eqref{eq:main_assumption_kappa}}
In order to understand \eqref{eq:main_assumption} and \eqref{eq:main_assumption_kappa}, we  
study situations, when these inequalities hold or fail; in particular for the case of random noise.

In this section, we specialize to  the case when $T$ is a compact operator, thus it allows for a  singular
system $\lambda_k, v_k, u_k,$ i.e., $\lambda_k>0$, 
$Tv_k=\lambda_k u_k, \ T^*u_k=\lambda_k v_k. $ 
Then
\eqref{eq:main_assumption} and \eqref{eq:main_assumption_kappa} can be equivalently  rephrased as 
\begin{align}
\exists C: \forall n\geq 1 \quad \lambda_n^4\sum_{k=1}^n \lambda_k^{-2}\lg y-y^\delta, u_k\rg^2
& \leq C \sum_{k=n+1}^\infty \lambda_k^2\lg y-y^\delta, u_k\rg^2  \label{eq:main_assumption1}
%\ee
\intertext{and} 
%\be
\exists C: \forall n\geq 1 \quad  \lambda_n^4\sum_{k=1}^n \lambda_k^{-2}\vk^2(\lambda_k^2)\lg y-y^\delta, u_k\rg^2
&\leq C \sum_{k=n+1}^\infty \lambda_k^2 \vk^2(\lambda_k^2)\lg y-y^\delta, u_k\rg^2,
\label{eq:main_assumption1_kappa}
\end{align}
respectively.

As an example, we now assume a 
 polynomially decaying deterministic  noise, i.e., 
\begin{equation}\label{polynoise}
\lg y-y^\delta, u_k\rg^2\asymp \delta^2 k^{-\nalp}, \qquad \nalp >0.
\end{equation}
%\mycomment{ I changed $\alpha$ in the exponent to $\nalp$ because $\alpha$ is the reg. parameter}
Then,  the following tables exemplify  some sufficient conditions for 
the noise condition  \eqref{eq:main_assumption1} for different degrees of ill-posedness:
\begin{center}
\begin{tabular}{cc|c|c}
%Assumption \eqref{eq:main_assumption1} holds &  Assumption \eqref{eq:main_assumption1_kappa} \\ \hline 

\multicolumn{2}{c|}{Ill-posedness}& noise & \begin{tabular}{c} sufficient condition \\
           for  \eqref{eq:main_assumption1} \end{tabular} \\ \hline
%
%Assumption \eqref{eq:main_assumption1_kappa}
mildly\rule{0mm}{5mm}&
$\lambda_k^2\asymp k^{-\beta}$ &  \eqref{polynoise} 
%$  \lg y-y^\delta, u_k\rg^2\asymp \delta^2 k^{-\nalp}$,
& 
$\nalp>1, \beta > \nalp-1,$  \\[2mm]  \hline
\begin{tabular}{c} 
severely\\ $a \in (0,1)$ \end{tabular}
& $\lambda_k^2\asymp a^k,$  &
\eqref{polynoise} 
%$\lg y-y^\delta, u_k\rg^2\asymp \delta^2 k^{-\nalp},$ 
& $\nalp>1$.
\end{tabular}
\end{center}

A similar results can be stated for the 
modified noise condition \eqref{eq:main_assumption1_kappa}: 

\begin{center}
\begin{tabular}{cc|c|c|c}
%Assumption \eqref{eq:main_assumption1} holds &  Assumption \eqref{eq:main_assumption1_kappa} \\ \hline 

\multicolumn{2}{c|}{Ill-posedness}& noise & $\vk$&  \begin{tabular}{c} sufficient condition \\
           for  \eqref{eq:main_assumption1_kappa} \end{tabular} \\ \hline
%
%Assumption \eqref{eq:main_assumption1_kappa}
mildly\rule{0mm}{5mm}&
$\lambda_k^2\asymp k^{-\beta}$ &  \eqref{polynoise}
%$  \lg y-y^\delta, u_k\rg^2\asymp \delta^2 k^{-\nalp}$,
& 
\begin{tabular}{c}
$\vk(t)\asymp t^{\gamma}$
\\  (or  $\vk^2(\lambda_k^2) \asymp k^{- 2\gamma\beta})$ \end{tabular} 
&
\begin{tabular}{c}
$\nalp>1, \gamma>0,$ \\
$\beta >2 \gamma\beta+ \nalp-1,$ \end{tabular}
\\[2mm] \hline
\begin{tabular}{c} 
severely\\ $a \in (0,1)$ \end{tabular}
& $\lambda_k^2\asymp a^k,$  &
\eqref{polynoise}
&  $\vk(t)\asymp t^{\gamma}$ &
$\nalp>1,\gamma \in (0,1)$. \\[2mm]  \hline 
\begin{tabular}{c} 
severely\\ $a \in (0,1)$ \end{tabular}
& $\lambda_k^2\asymp a^k,$ &\eqref{polynoise}
%$\lg y-y^\delta, u_k\rg^2\asymp \delta^2 k^{-\nalp},$ 
& $\vk(t)\asymp  (\log{t^{-1}})^{-\gamma}$& $\nalp>1,\gamma >0$.
\end{tabular}
\end{center}

In contrast to the deterministic case, we now investigate 
the case of random noise. 
We assume  that the noise is random and  of the form
 \be\label{eq:stochastic_noise}
y^\delta-y=\sum_{k=1}^\infty \sigma_k(\delta) \xi_k u_k,
\ee
where  $ \xi_k=\xi_k(\omega), \omega\in \Omega$ are   
independent random variables given on a probability space $(\Omega, \cF, \PP)$, with 
\begin{equation}\label{eqvar}
\E \xi_k=0,  \qquad \Var(\xi_k)=1,
\end{equation}
and analogous to \eqref{polynoise},  we assume that
\begin{equation}\label{eq:asumpt_sigma}
\sigma^2_k(\delta)\asymp \delta^2 k^{-\nalp}, \qquad \nalp>1, \end{equation} 
Note that $\E (y^\delta-y)=0$ and $\Var(y^\delta-y)\asymp \delta^2.$

The stochastic analogue of the   inequality 
\eqref{eq:main_assumption1} is of the following  form:
For almost all $\omega$ there is a constant $C=C(\omega)$ such that 
\be\label{eq:main_assumption2}
\forall n\geq 1 \qquad  \lambda_n^4\sum_{k=1}^n \lambda_k^{-2}  \sigma_k^2(\delta)\xi_k^2 
\leq C \sum_{k=n+1}^\infty \lambda_k^2 \sigma_k^2(\delta)\xi_k^2.
\ee
or
\begin{equation*}%\label{eq:main_assumption3}
\sup_{n\geq 1}\frac{\lambda_n^4\sum_{k=1}^n \lambda_k^{-2}  \sigma_k^2(\delta)\xi_k^2 }{  \sum_{k=n+1}^\infty \lambda_k^2 \sigma_k^2(\delta)\xi_k^2}<\infty
\ \ \mbox{almost surely}.
\end{equation*}
The stochastic analogue of \eqref{eq:main_assumption1_kappa} can be considered 
similarly with the natural modifications.

\begin{thm}\label{thm:mildly}
Assume a mildly ill-posed case, i.e., $\lambda_k^2 \asymp k^{-\beta}$, with $\beta >0$. 
Moreover, let the noise satisfy  \eqref{eq:stochastic_noise}--\eqref{eq:asumpt_sigma}, and 
assume that the random variables $\{\xi_k\}$ have moments of all orders:
\[
\forall p\geq 1\ \ \ \sup_k\E |\xi_k|^p<\infty.
\]
Then, if $\beta > \nalp -1$
\be\label{eq:main_assumption3}
\sup_{n\geq 1}\frac{\lambda_n^4\sum_{k=1}^n \lambda_k^{-2}  \sigma_k^2(\delta)\xi_k^2 }{  \sum_{k=n+1}^\infty \lambda_k^2 \sigma_k^2(\delta)\xi_k^2}<\infty
\ \ \mbox{almost surely}.
\ee
\end{thm}
The proof of this theorem is given below.
The assumptions on $\{\xi_k\}$ hold in particular for 
independent Gaussian $N(0,1)$-random variables. 
Thus, for the mildly ill-posed operators, the stochastic case is completely similar to the 
deterministic one and the analogous convergence rates results hold true (almost surely).

This, however, is not true for the severely ill-posed case as the following theorem shows. 
\begin{thm}\label{thm:severe}
Assume a severely ill-posed case, i.e., $\lambda_k^2 \asymp a^{k}$, with $a \in (0,1)$ and 
let 
\eqref{eq:stochastic_noise} and \eqref{eq:asumpt_sigma} hold, where 
 $\{\xi_k\}$ are  independent Gaussian $N(0,1)$ random variables.
Then
\be\label{eq:main_assumption3_ex1}
\PP\left(\sup_{n\geq 1} \frac{\lambda_n^4\sum_{k=1}^n 
\lambda_k^{-2}  \sigma_k^2(\delta)\xi_k^2 }{  \sum_{k=n+1}^\infty \lambda_k^2 
\sigma_k^2(\delta)\xi_k^2}=\infty\right)=1.
%
%\PP\left(\sup_{n\geq 1}\frac{a^{2n}\sum_{k=1}^n  a^{-k}  
%k^{-\nalp} \xi_k^2 }{  \sum_{k=n+1}^\infty a^k  k^{-\nalp} \xi_k^2}=\infty\right)=1.
\ee
\end{thm}
In particular, in this situation, the noise condition \eqref{eq:main_assumption1} fails almost surely. 
This shows that  the difference between stochastic and deterministic cases may be very essential.

\begin{proof}[Proof of Theorem~\ref{thm:severe}.]

Introduce the Markov moment $\tau_p:=\inf\{ n\geq 1\ : \ \xi_n^2>p\}.$ Obviously,
$\tau_p<\infty$ almost surely.
Then 
\[
a^{2\tau_p}\sum_{k=1}^{\tau_p}  a^{-k}  k^{-\nalp} \xi_k^2\geq a^{2\tau_p}   a^{-{\tau_p}}  {\tau_p}^{-\nalp} p= a^{\tau_p}  {\tau_p}^{-\nalp}p.
\]
We have 
\[
\sum_{k=\tau_p+1}^\infty a^k  k^{-\nalp} \xi_k^2= a^{\tau_p} \sum_{k=1}^\infty a^k  (k+{\tau_p})^{-\nalp} \xi_{k+{\tau_p}}^2
\leq a^{\tau_p} {\tau_p}^{-\nalp}\sum_{k=1}^\infty a^k    \xi_{k+{\tau_p}}^2.
\]
Since $\tau_p$ is a finite  Markov moment, 
\[
\E \xi_{k+{\tau_p}}^2 =\sum_{n\geq 0} \E(\1_{\tau_p=n} \xi_{k+n}^2)=
\sum_{n\geq 0} \E \1_{\tau_p=n} \E\xi_{k+n}^2  = \sum_{n\geq 0} \E\1_{\tau_p=n} =1.
\]
Hence $\E\sum_{k=1}^\infty a^k    \xi_{k+{\tau_p}}^2= \sum_{k=1}^\infty a^k=(1-a)^{-1}.$
By Chebyshev's inequality we have
\[
\PP( \sum_{k=1}^\infty a^k    \xi_{k+{\tau_p}}^2\geq \sqrt{p})\leq ((1-a)p)^{-1/2}.
\] 
Therefore for any $p\geq 1$
\begin{align*}
& \PP\left(\exists{n\geq 1}\ \ \ \frac{a^{2n}\sum_{k=1}^n  a^{-k}  k^{-\nalp} \xi_k^2 }{  \sum_{k=n+1}^\infty a^k  k^{-\nalp} \xi_k^2}
 \geq \sqrt{p}\right)\\
 &
\PP\left(  \frac{ a^{\tau_p} {\tau_p}^{-\nalp}p}{a^{\tau_p} {\tau_p}^{-\nalp}{\sum_{k=1}^\infty a^k    \xi_{k+{\tau_p}}^2}}\geq \sqrt{p}
 \right)\\
 &=
\PP\left(  \frac{  p}{ {\sum_{k=1}^\infty a^k    \xi_{k+{\tau_p}}^2}}\geq \sqrt{p}
 \right) = \PP\left(   { {\sum_{k=1}^\infty a^k    \xi_{k+{\tau_p}}^2}}\leq \sqrt{p}
 \right) \geq 1-{((1-a)p)}^{-1/2}.
\end{align*}
This yields \eqref{eq:main_assumption3_ex1}.
\end{proof}

In this proof  we used the following properties of the sequence $\{\xi_k\}$:
\[ 
\mbox{ i)  independence,}  \qquad 
\mbox{ ii) } \sup_{k}\E\xi_k^2<\infty, \qquad 
\mbox{ iii) } \limsup_{k\to\infty}|\xi_k|=+\infty  \mbox{ a.s.}  
\]

\begin{remk}\label{rem1}
It may be conjectured that if $\{\xi_k\}$ are uniformly bounded random variables, 
for example, if $\{\xi_k\}$ have the uniform distribution on $[-1,1]$, then
\eqref{eq:main_assumption3} would hold. 
However, this conjecture is wrong. Problems may arise if $\{\xi_k\}$ are i.i.d.
and  0 belongs to the support of the $\xi_k$'s distribution, i.e., if  $\PP(|\xi_k|<\ve)>0$ for any $\ve>0.$

Indeed, let $m,p\geq 1$ be fixed. Select  $c>0$   such that $\PP(|\xi_k|>c)>0.$ 
Set 
\begin{align*}
\tau_{mp}&:=\inf\bigg \{n\geq 1\ : \ \ |\xi_{n-p}|>c, |\xi_{n-m+1}|<p^{-1},
 |\xi_{n-m+2}|<p^{-1},\ldots \\
 & \qquad \qquad \qquad \ldots, |\xi_{n-1}|<p^{-1}, |\xi_{n}|<p^{-1}\bigg\}
\end{align*}
Since $\PP(|\xi_k|>c)>0 $ and $\PP(|\xi_k|<p^{-1})>0$, the 
random variable $\tau_{mp} $ is finite almost surely.

Similarly to the reasoning above we get the 
 inequalities \allowdisplaybreaks
\begin{align*}
a^{2(\tau_{mp}-m) }\sum_{k=1}^{\tau_{mp}-m}  a^{-k}  k^{-\nalp} \xi_k^2&\geq a^{2\tau_{mp}-m} 
  a^{-{\tau_{mp}-m}}  {\tau_{mp}}^{-\nalp} c= a^{\tau_{mp}}  ({\tau_{mp}}-m)^{-\nalp}c,
\\
\sum_{k=\tau_{mp}-m+1}^\infty a^k  k^{-\nalp} \xi_k^2&\leq
 a^{\tau_{mp}-m} ({\tau_{mp}-m})^{-\nalp}(m/p+a^m\sum_{k=1}^\infty a^k    \xi_{k+{\tau_{mp}}}^2).
\end{align*}
Chose $m\in(-\frac{\log p}{\log a}, \sqrt{p}/2 )$, i.e., $m/p< 1/(2\sqrt{p})$ and $a^m<1/p$. Then  
\begin{align*}
 &\PP\left(\exists{n\geq 1}\ \ \ \frac{a^{2n}\sum_{k=1}^n  a^{-k}  k^{-\nalp} \xi_k^2 }{  \sum_{k=n+1}^\infty a^k  k^{-\nalp} \xi_k^2}
 \geq c\sqrt{p}\right)\\
 &\geq 
 \PP\left(  \frac{a^{\tau_{mp}}  ({\tau_{mp}}-m)^{-\nalp}c}{  a^{\tau_{mp}-m} ({\tau_{mp}-m})^{-\nalp}(m/p+a^m\sum_{k=1}^\infty a^k    \xi_{k+{\tau_{mp}}}^2)} 
 \geq c\sqrt{p}\right)\\
 &\geq 
 \PP\left(  {  1/(2\sqrt{p})+1/p\sum_{k=1}^\infty a^k    \xi_{k+{\tau_{mp}}}^2 }  
 \leq 1/\sqrt{p}\right)\\
 &\geq \PP\left(  {    \sum_{k=1}^\infty a^k    \xi_{k+{\tau_{mp}}}^2 }  
 \leq  \sqrt{p}/2\right)\to 1, \quad \mbox{ as } p\to\infty.
\end{align*}
and we  again obtain \eqref{eq:main_assumption3_ex1}, the failure of the noise condition.

The conclusion from the above reasoning is that  if $\{\xi_k\}$ are i.i.d. and 
$
\lambda_k^2\asymp a^k$ where $a\in(0,1),$ then assumption 
\eqref{eq:main_assumption3} is true if  $\PP(|\xi_k|\in [\ve, \ve^{-1}])=1$ for
some $\ve>0,$ that is, the support of $\xi_k$
is separated from 0 and $\infty.$
The sufficiency follows from the deterministic statement.
\end{remk}
%\mycomment{Before we had ``if and only if''. I deleted the ``only if'' because I am not sure 
%about it. It could  happen that  $\PP(|\xi_k|\in [\ve, \ve^{-1}])=1$ is not true for 
%a finite number of $k$ but the condition \eqref{eq:main_assumption3} still holds. 
%We may discuss this.}
%\end{expl}

% Here we are 

To prove the positive results in the mildly ill-posed case, 
we need the following known result. 
\begin{lem}\label{lem:almost_sure}
Assume that random variables $\{Y_n\}$ have the finite second moment and
\[
\lim_{n\to\infty} \E Y_n=0,\ \ \sum_{n=1}^\infty \Var{Y_n}<\infty.
\]
Then 
\be\label{eq1020}
Y_n\to0,  \mbox{ as }  n\to \infty \ \mbox{ almost surely.}
\ee
\end{lem}
\begin{proof}
Indeed,
\[
\E \sum_n(Y_n-\E Y_n)^2= \sum_n\E(Y_n-\E Y_n)^2=\sum_n\Var Y_n<\infty.
\]
So, $Y_n-\E Y_n\to0$ as $n\to\infty$ almost surely, and   we get \eqref{eq1020} because  $\lim_{n\to\infty} \E Y_n=0.$
\end{proof}

\begin{proof}[Proof of Theorem~\ref{thm:mildly}]
We show \eqref{eq:main_assumption3} if  $\beta>\nalp-1$.
In particular, \eqref{eq:main_assumption1} is a particular case of 
\eqref{eq:main_assumption3} if $\PP(\xi_k=\pm1)=1/2.$

To prove  \eqref{eq:main_assumption3}, it suffices to verify that
\be\label{eq:main_assumption_ex2}
\PP\left(
\sup_{n\geq 1}\frac{ n^{-2\beta}\sum_{k=1}^n  k^{\beta}   k^{-\nalp} \xi_k^2 }{  \sum_{k=n+1}^\infty  k^{-\beta}  k^{-\nalp} \xi_k^2}<\infty\right)=
\PP\left(
\sup_{n\geq 1}\frac{ n^{-2\beta}\sum_{k=1}^n  k^{\beta-\nalp} \xi_k^2 }{  \sum_{k=n+1}^\infty  k^{-\beta-\nalp} \xi_k^2}<\infty\right)=1.
\ee

Set $\eta_k=\xi_k^2-1.$ Recall that $\E \eta_k=0.$
We have 
\begin{align*} 
&n^{-2\beta}\sum_{k=1}^n  k^{\beta-\nalp} \xi_k^2 
=n^{-2\beta}\sum_{k=1}^n  k^{\beta-\nalp} (1+(\xi_k^2-1))\\
&\quad =
n^{-2\beta}\sum_{k=1}^n  k^{\beta-\nalp} + 
n^{-2\beta}\sum_{k=1}^n  k^{\beta-\nalp}\eta_k
  = \frac{ n^{-\beta-\nalp+1}}{\beta-\nalp+1} (1+o(1))+ 
n^{-2\beta}\sum_{k=1}^n  k^{\beta-\nalp}\eta_k.
\end{align*}
\begin{align*}
&\sum_{k=n+1}^\infty  k^{-\beta-\nalp} + \sum_{k=n+1}^\infty  k^{-\beta-\nalp}\eta_k
 =
   \frac{n^{-\beta-\nalp+1}}{-\beta-\nalp+1} (1+o(1))
+ \sum_{k=n+1}^\infty  k^{-\beta-\nalp}\eta_k.
\end{align*}

So, equation  \eqref{eq:main_assumption_ex2} will be verified  if we prove that
\be\label{eq:828}
n^{ \beta+\nalp-1}n^{-2\beta}\sum_{k=1}^n  k^{\beta-\nalp}\eta_k=
 n^{ -\beta+\nalp-1} \sum_{k=1}^n  k^{\beta-\nalp}\eta_k \to 0 \quad  \text{ as }  n\to\infty \ \ \mbox{almost surely}
\ee
and 
\be\label{eq:833}
n^{ \beta+\nalp-1}\sum_{k=n+1}^\infty  k^{-\beta-\nalp}\eta_k\to 0 \quad  \text{ as } n\to\infty \ \ \mbox{almost surely}.
\ee

Consider \eqref{eq:833}. Set 
$Y_n:=\left(n^{ \beta+\nalp-1}\sum_{k=n+1}^\infty  k^{-\beta-\nalp}\eta_k\right)^2$.
Since $\E \eta_k=0,$ we have
\begin{align*} 
\E Y_n&= \E\left(n^{ \beta+\nalp-1}\sum_{k=n+1}^\infty  k^{-\beta-\nalp}\eta_k\right)^2
=\Var\left(n^{ \beta+\nalp-1}\sum_{k=n+1}^\infty  k^{-\beta-\nalp}\eta_k\right) \\
&= n^{ 2\beta+2\nalp-2}\sum_{k=n+1}^\infty \Var (k^{-\beta-\nalp}\eta_k) =
 n^{ 2\beta+2\nalp-2}\sum_{k=n+1}^\infty k^{-2\beta-2\nalp} \Var (\eta_k)\\
 &\leq
C_1 n^{ 2\beta+2\nalp-2}\sum_{k=n+1}^\infty k^{-2\beta-2\nalp} \
\leq C_2 n^{ 2\beta+2\nalp-2}n^{-2\beta-2\nalp+1} \\
&= C_2 n^{-1} \to0,\ \mbox{ as } n\to\infty.
\end{align*} \allowdisplaybreaks
\begin{align*}
&\Var Y_n= \Var\left[\left(n^{ \beta+\nalp-1}\sum_{k=n+1}^\infty  k^{-\beta-\nalp}\eta_k\right)^2\right]\\
&=n^{ 4\beta+4\nalp-4}\Var\left[\sum_{k=n+1}^\infty  k^{-2\beta-2\nalp}\eta_k^2 + 
2\sum_{n+1\leq i <j}i^{-\beta-\nalp}j^{-\beta-\nalp}\eta_i\eta_j  \right]\\
& \leq 
n^{ 4\beta+4\nalp-4}\Var\left[\sum_{k=n+1}^\infty  k^{-2\beta-2\nalp}(\eta_k^2-\E \eta_k) + 
2\sum_{n+1\leq i <j}i^{-\beta-\nalp}j^{-\beta-\nalp}\eta_i\eta_j  \right]\\
&=
n^{ 4\beta+4\nalp-4}\E\left[\sum_{k=n+1}^\infty  k^{-2\beta-2\nalp}(\eta_k^2-\E \eta_k) + 
2\sum_{n+1\leq i <j}i^{-\beta-\nalp}j^{-\beta-\nalp}\eta_i\eta_j  \right]^2\\
&\leq
2n^{ 4\beta+4\nalp-4}\left(
\E\left[\sum_{k=n+1}^\infty  k^{-2\beta-2\nalp}(\eta_k^2-\E \eta_k^2)\right]^2 + 
\E\left[2\sum_{n+1\leq i <j}i^{-\beta-\nalp}j^{-\beta-\nalp}\eta_i\eta_j  \right]^2
\right).
\end{align*}
If we expand the brackets in the last sum, then
 the expectation $\E (\eta_i\eta_j\eta_{i_1}\eta_{j_1})$ is equal to  zero if $(i,j)\neq (i_1,j_1)$ and $(i,j)\neq (j_1,i_1)$.
Thus the right hand side of the last expression equals
\begin{align*}
& n^{ 4\beta+4\nalp-4}\left(
\Var\left[\sum_{k=n+1}^\infty  k^{-2\beta-2\nalp}(\eta_k^2-\E \eta_k^2)\right] + 
 8\sum_{n+1\leq i<j}\E\left[i^{-\beta-\nalp}\eta_i\right]^2\E\left[j^{-\beta-\nalp}\eta_j  \right]^2
\right)\\
&=
   n^{ 4\beta+4\nalp-4}\Bigg(
\sum_{k=n+1}^\infty  k^{-4\beta-4\nalp}\Var\left[(\eta_k^2-\E \eta_k^2)\right] \\
& \qquad \qquad \qquad \qquad \qquad \qquad + 
  8\sum_{n+1\leq i<j} \Var\left[i^{-\beta-\nalp}\eta_i\right] \Var\left[j^{-\beta-\nalp} \eta_j  \right]
\Bigg)\\
&\leq
C_3   n^{ 4\beta+4\nalp-4}\left(
\sum_{k=n+1}^\infty  k^{-4\beta-4\nalp}+
  \sum_{n+1\leq i<j}  i^{-2\beta-2\nalp}\Var(\eta_i)  j^{-2\beta-2\nalp} \Var(\eta_j) 
\right)\\
&\leq
C_4   n^{ 4\beta+4\nalp-4}\left(
n^{-4\beta-4\nalp+1}+
   (\sum_{n+1\leq i }  i^{-2\beta-2\nalp})^2 \right)\\
   &\leq
   C_5   n^{ 4\beta+4\nalp-4}\left(
n^{-4\beta-4\nalp+1}+
   (n^{-2\beta-2\nalp+1})^2 \right)\leq C_6  n^{ 4\beta+4\nalp-4}  n^{-4\beta-4\nalp+2}=C_6  n^{-2}.
\end{align*}
This proves \eqref{eq:833}.

Consider \eqref{eq:828}.
Set $Y_n:=n^{ -2\beta+2\nalp-2} (\sum_{k=1}^n  k^{\beta-\nalp}\eta_k )^2$
in Lemma \ref{lem:almost_sure}. Similarly to the above calculations we get $\lim_{n\to\infty}\E Y_n=0$ and
\begin{equation}
\begin{split} 
%\[%\be\label{eq:1016}
\Var(Y_n)&
=n^{ -4\beta+4\nalp-4}\; O\left(  \sum_{k=1}^n  k^{4\beta-4\nalp} + 
(\sum_{k=1}^n  k^{2\beta-2\nalp})^2 \right)\\
%\]%\ee
%\be
\label{eq:1016}
&=n^{ -4\beta+4\nalp-4}\; O\left( \left( \sum_{k=1}^n  k^{2\beta-2\nalp}\right)^2\right).
\end{split}
\end{equation}
In contrast to (convergent) sums of the form $\sum_{k=n+1}^\infty  k^{-\theta}\asymp n^{-\theta+1}$, the asymptotic
of $\sum_{k=1}^n  k^{-\theta}$ is different:
\[
\sum_{k=1}^n  k^{-\theta}\asymp
\begin{cases}
 n^{-\theta+1}, & \theta<1;\\
 \log n, & \theta =1;\\
 1=n^0, & \theta>1.
\end{cases}
=
\begin{cases}
 n^{(-\theta+1)\vee 0}, & \theta\neq 1;\\
 \log n, & \theta =1. 
 \end{cases}
\]
That's why, we have to be careful in \eqref{eq:1016}. In any case,
$\lim_{n\to\infty}  \Var(Y_n)=0$ and the series $\sum_n \Var(Y_n) $ is convergent if $-4\beta+4\nalp-4<-1$ or $\beta> \nalp-1+\frac14=\nalp-\frac34.$
Thus, we  have already proved  \eqref{eq:828} for $\beta>  \nalp-\frac34$. To verify \eqref{eq:828} for
$\beta>  \nalp-1$,
we have to consider moments of higher orders.

Considering 
$Y_n^{(m)}:=n^{ -2m\beta+2m\nalp-2m} (\sum_{k=1}^n  k^{\beta-\nalp}\eta_k )^{2m}$ and performing
similar calculation as above we get
\begin{align*}%\be\label{eq:1016}
\Var(Y_n^{(m)})=n^{ -4m\beta+4m\nalp-4m}\; O\bigg(& \sum_{k_1+...+k_p=2m,\ k_i\geq 2} \
\sum_{i_1=1}^n  i_2^{k_1(\beta- \nalp)}\sum_{i_2=1}^n  i_2^{k_2(\beta- \nalp)}\ldots \\
&  \qquad \qquad   \ldots \sum_{i_p=1}^n  i_p^{k_p(\beta- \nalp)}\bigg).
\end {align*}%\ee
It can be seen that
$\sum_n \Var{Y_n^{(m)}}<\infty$ if $-4m\beta+4m\nalp-4m>-1$ or $\beta> \nalp-1+\frac1m$. So, 
we have \eqref{eq:828} for  $\beta>  \nalp-1+\frac1m.$
Since $m\geq 1$ is arbitrary, this yields  \eqref{eq:828} for  $\beta>  \nalp-1.$

\end{proof}
%\end{expl}

\begin{remk}
It is interesting that the Muckenhoupt-type condition fails for a typical random noise in the case of 
severely ill-posed problems. This observation, however, is in line with numerical investigation 
on the performance of heuristic rules done, for instance, by H\"amarik, Palm, and Raus 
\cite{HaPaRe11}, in particular in \cite{Palm10}. Typically, for mildly ill-posed problems, 
 the quasi-optimality principle is amongst  the most efficient heuristic rules. However, for the   
backward heat equation (which is severely ill-posed), it performs worse 
compared to competitors such as the Hanke-Raus rules which by our results can be understood 
as caused by the failure of the noise condition.  Note that the convergence theory 
for the latter rules is based on a weaker Muckenhoupt-type condition which might not suffer 
from the negative result in Theorem~\ref{thm:severe}. Thus, the restricted noise analysis 
clearly reveals the behaviour of heuristic rules, which was quite mysterious for a long time.  
\end{remk}

%\mycomment{I added the above remark}

%%%%%%%%%%%%%%%%%%%%%%%%%%%%%%%%%%%%%%%%%%%%%%%%%%%%%%%%%%%%%%%%%%%%%%%%%%%%%%%%%%%%%%%%%%%%%%%%%%%%%%%%%%%%%%%%%%%%%%%%%%%%%%%%%

\section{The quasi-optimality criterion in the aggregation of the regularized approximants:
numerical illustration}\label{section:numerical}

%%%%%%%%%%%%%%%

In this section, we illustrate how the quasi-optimality criterion can be used in the aggregation of the
regularized approximants by means of the linear functional strategy. Recall that the idea of such an aggregation
is to approximate the best linear combination
$$
  x_{ \tm{agg} }^{s}  = \sum_{j=1}^s c_j^s x_{\alp_j}^{\dlt}
$$
of the constructed regularized approximants $x_{\alp_j}^{\dlt}$ of $x$, where ``best'' means that
$x_{ \tm{agg} }^{s}$ solves the minimization problem
$$
  \norm{  x - x_{ \tm{agg} }^{s}  } = \min\limits_{c_j}
  \norm{   x - \sum_{j=1}^s c_j x_{\alp_j}^{\dlt}   }.
$$
%

%%%%%%%%%%%%%%%

It is clear that the vector $\cv^s = \kl{   c_1^s,c_2^s, \ldots, c_s^s   }  \in\R^s  $ satisfies the system of linear equations
$\GGm \cv = \pv$ with the Gram matrix $  \GGm = \kl{   \inner{  x_{\alp_i}^{\dlt},  x_{\alp_j}^{\dlt}  }:\; 
i,j = 1,2,\ldots,s      }  $
and the vector $  \pv = \kl{   \inner{  x,  x_{\alp_i}^{\dlt}  }  :\; i = 1,2,\ldots,s      }  $.
Since $x_{\alp_j}^{\dlt}$, $j=1,2,\ldots,s$, are already found, the matrix $\GGm$ can be computed and the calculation
of the inverse matrix $\GGm^{-1}$ can be controlled. However, the vector $\pv$ involves the unknown solution $x$,
and therefore, the system $\GGm \cv = \pv$ cannot be solved directly.

At the same time, each component $\inner{  x,  x_{\alp_i}^{\dlt}  }$ of the vector $\pv$ is a value of a bounded
linear functional $x_{\alp_i}^{\dlt}$, and the linear functional strategy allows us to estimate $\inner{  x,  x_{\alp_i}^{\dlt}  }$,
$i=1,2,\ldots,s$, more accurately than $x$ in $\norm{ \cdot }$. For example, if $ x\in \Ran \kl{   \vph\kl{   T^* T   }   }  $
and $  x_{\alp}^{\dlt} = \kl{  \alp I + T^* T  }^{-1} T^* y^{\dlt}  $, then under the conditions of Theorem~\ref{thm:main}, we have
\begin{equation}\label{n5:1}
   \norm{  x  -  x_{ \alp\kl{  y^{\dlt}  }  }^{\dlt}  }   =
  O\kl{\:   \vph\kl{  \vph\kl{  \tht^{-1} (\dlt)  }  }     \:},
\end{equation}
while for each $\alp_i$, the quasi-optimality criterion in the linear functional strategy gives us
$\alp_i\kl{ y^{\dlt} } = \alp_{ \vk_i } \kl{ y^{\dlt} } $ such that
\begin{equation}\label{n5:2}
   \abs{ \inner{  x, x_{\alp_i}^{\dlt}  } - \inner{  x_{ \alp_i\kl{  y^{\dlt}  }  }^{\dlt}  , x_{\alp_i}^{\dlt}  }  }   =
  o\kl{\:   \vph\kl{  \vph\kl{  \tht^{-1} (\dlt)  }  }     \:},
\end{equation}
where $\vk_i$ is an index function for which
$x_{\alp_i}^{\dlt} \in \Ran\kl{   \vk_i\kl{   T^* T   }   }   $.

%%%%%%%%%%%%%%%

Consider now $$ \pv_{ y^{\dlt} }  = \kl{   \inner{    x_{ \alp_i\kl{  y^{\dlt}  }  }^{\dlt}   , 
x_{\alp_i}^{\dlt}     },\;   i=1,2,\ldots,s    } , \quad 
  \cv_{  y^{\dlt}  }^s = \kl{  c_{  1,y^{\dlt}  }^s,  c_{  2,y^{\dlt}  }^s, \ldots,  c_{  s,y^{\dlt}  }^s     }  
  = \GGm^{-1}  \pv_{ y^{\dlt} }   $$ and
\begin{equation}\label{n5:3}
  x_{ \tm{agg}, y^{\dlt} }^{s} = \sum\limits_{j=1}^{s}  c_{  j,y^{\dlt}  }^s  x_{\alp_j}^{\dlt}.
\end{equation}
Note that $  x_{ \tm{agg}, y^{\dlt} }^{s}  $ can be effectively computed because it only uses access to $T$ and $y^{\dlt}$.
Then by the same arguments as in the proof of Theorem~3.7 in~\cite{ChePerXu15}, it follows from~\eqref{n5:2} that
\begin{equation}\label{n5:4}
\begin{aligned}
   \norm{  x - x_{ \tm{agg}, y^{\dlt} }^{s}  } & =
	 \min\limits_{c_j} \norm{  x - \sum_{j=1}^s  c_j x_{\alp_j}^{\dlt}  }
	 +  o\kl{\:   \vph\kl{  \vph\kl{  \tht^{-1} (\dlt)  }  }     \:}
\\
 &= \norm{  x - x_{ \tm{agg} }^{s}  }  +  o\kl{\:   \vph\kl{  \vph\kl{  \tht^{-1} (\dlt)  }  }     \:}.
\end{aligned}
\end{equation}
%

%%%%%%%%%%%%%%%

If
\begin{equation}\label{n5:5}
  \alp\kl{  y^{\dlt}  } \in \Set{    \alp_j,\; j=1,2,\ldots,s    },
\end{equation}
then the accuracy of $x_{ \tm{agg} }^{s}$ may only be better than the one of $x_{ \alp\kl{  y^{\dlt}  }  }^{\dlt}$.
Moreover, from~\eqref{n5:1}, \eqref{n5:4}, it follows that the error of the effectively computed aggregator
$x_{ \tm{agg}, y^{\dlt} }^{s}$ differs from the error of $x_{ \tm{agg} }^{s}$ by a quantity of higher order
than the accuracy guaranteed by the standard quasi-optimality criterion. In this way, a combination
of the linear functional strategy and the quasi-optimality criterion resulting in~\eqref{n5:3} may improve
the accuracy of the latter one. Such improvement indeed is observed in the numerical illustrations below.

Note that the family of the regularized approximations $\Set{  x_{\alp_j}^{\dlt}  }$ may consist only of a single
approximant $  x_{\alp_i}^{\dlt}  $. Then the value of
$$
    c_i^* = \amin\limits_{c} \norm{  x - c x_{\alp_i}^{\dlt}  }
$$
can be explicitly written as
$$
    c_i^* = \frac{    \inner{   x,  x_{\alp_i}^{\dlt}    }    }
    {  \norm{ x_{\alp_i}^{\dlt}  }^2   },
$$
and can be interpreted as a correction factor for $  x_{\alp_i}^{\dlt}  $. If a value $\alp = \alp\kl{   y^{\dlt}   }$ has been already
selected by the quasi-optimality criterion, then $c_i^*$ can be approximated by
\begin{equation}\label{n5:6}
    c_{i,y^{\dlt}  } = \frac{    \inner{   x_{ \alp\kl{  y^{\dlt}  }  }^{\dlt},  x_{\alp_i}^{\dlt}    }    }
    {  \norm{ x_{\alp_i}^{\dlt}  }^2   },
\end{equation}
and under the conditions of Theorem~\ref{thm:main}, we have
$$
   \abs{ c_i^* -  c_{i,y^{\dlt}  }   } =
   o\kl{\:   \vph\kl{  \vph\kl{  \tht^{-1} (\dlt)  }  }     \:}.
$$
%

%%%%%%%%%%%%%%%

After calculating~\eqref{n5:6} for each considered $\alp_i$, we can construct a corrected family of regularized approximants
$\Set{    \bar{x}_{\alp_i}^{\dlt}  =   c_{i,  y^{\dlt} }    x_{\alp_i}^{\dlt}     }  $
such that
$$
     \norm{    x    -   \bar{x}_{\alp_i}^{\dlt}   } = \min\limits_{c}
     \norm{    x    -     c {x}_{\alp_i}^{\dlt}   } + o\kl{\:   \vph\kl{  \vph\kl{  \tht^{-1} (\dlt)  }  }     \:}.
$$
If~\eqref{n5:5} is satisfied, then by the same reason as above, the corrected family
$   \Set{   \bar{x}_{\alp_i}^{\dlt}   }   $   may contain elements approximating $x$ better than
$x_{ \alp\kl{  y^{\dlt}  }  }^{\dlt}$ that suggests second or iterated application of the
quasi-optimality criterion, this time to the corrected family $\Set{   \bar{x}_{\alp_i}^{\dlt}    }$.
This iterated quasi-optimality criterion will also be illustrated below.

Recall that the usual way (see~\cite{TikGla65}) of implementing the quasi-optimality criterion consists in
selecting $\alp = \alp\kl{  y^{\dlt}  } = \alp_{\ell}  $ from a geometric sequence
\begin{equation}\label{n5:7}
  \Set{  \alp_j = \alp_1 q^{ j-1 },\; j=1,2,\ldots,M  },\;
	0< \alp_1,\; q<1,
\end{equation}
such that
\begin{equation}\label{n5:8}
  \norm{  x_{\alp_{\ell}}^{\dlt} - x_{\alp_{\ell-1}}^{\dlt}  } =
	\min\skl{  \norm{  x_{\alp_{j}}^{\dlt} - x_{\alp_{j-1}}^{\dlt}  },\;   j=2,3,\ldots,M   } .
\end{equation}
In the same spirit, we can implement the above mentioned iterated quasi-optimality criterion suggesting
$\alp = \bar{\alp}\kl{  y^{\dlt}  } = \alp_{ k }  $ such that
\begin{equation}\label{n5:9}
\begin{split}
  \norm{  \bar{x}_{\alp_{k}}^{\dlt} - \bar{x}_{\alp_{k-1}}^{\dlt}  } &=
	%%%\Set
	\min \Bigl\{  \norm{  \bar{x}_{\alp_{j}}^{\dlt} - \bar{x}_{\alp_{j-1}}^{\dlt}  }  \\
	&=\norm{  c_{j,y^{\dlt}} x_{\alp_{j}}^{\dlt} - c_{j-1,y^{\dlt}} x_{\alp_{j-1}}^{\dlt}  },
	\;   j=2,3,\ldots,M   \Bigr\}.
\end{split}
\end{equation}

Note that the rule~\eqref{n5:8} is in fact a discretization of the quasi-optimality criterion considered above
because $\psi\kl{  \alp, y^{\dlt}  }$ can be written (see, e.g., \cite{Leo91}) as
$$
  \psi\kl{  \alp, y^{\dlt}  } = \alp
	\norm{ \frac{ \partial x_{\alp}^{\dlt} }{ \partial \alp }  },
$$
and~\eqref{n5:8} is just a backward difference approximation of the derivative
$\frac{ \partial x_{\alp}^{\dlt} }{ \partial \alp }$ on the mesh nodes~\eqref{n5:7}, i.e.
\begin{equation}\label{n5:10}
   \left. \alp   \frac{ \partial x_{\alp}^{\dlt} }{ \partial \alp }   \right|_{   \alp = \alp_j  } \approx
   \alp_{j} \frac{ x_{\alp_{j}}^{\dlt} - x_{\alp_{j-1}}^{\dlt}  }
	 { \alp_{j} - \alp_{j-1} } = (q-1)^{-1} \kl{ x_{\alp_{j}}^{\dlt} - x_{\alp_{j-1}}^{\dlt}  }.
\end{equation}
From this view point, the iterated quasi-optimality criterion~\eqref{n5:9} can be seen as the use of another
difference formula to approximate $\frac{ \partial x_{\alp}^{\dlt} }{ \partial \alp }$, i.e.
\begin{equation}\label{n5:11}
   \left. \alp   \frac{ \partial x_{\alp}^{\dlt} }{ \partial \alp }   \right|_{   \alp = \alp_j  }   \approx
   \alp_{j} \frac{   c_{j, y^{\dlt} }  x_{\alp_{j}}^{\dlt} -     c_{j-1, y^{\dlt} }  x_{\alp_{j-1}}^{\dlt}  }
	 { \alp_{j} - \alp_{j-1} } = (q-1)^{-1} \kl{ \bar{x}_{\alp_{j}}^{\dlt} - \bar{x}_{\alp_{j-1}}^{\dlt}  }.
\end{equation}
%

%%%%%%%%%%%%%%%

The quasi-optimality criterion in the linear functional strategy is associated with the function
$\psi_{\vk}  \kl{   \alp, y^{\dlt}  }  $ that is a particular form of the quantity used in the so-called
weighted quasi-optimality criterion discussed in~\cite{BauRei08} (see Definition~2.5 there).
At the same time, $\psi_{\vk} \kl{   \alp,  y^{\dlt}   } $ is, up to a constant multiplier, the upper
bound for all functions
$$
   U_f \kl{  \alp,  y^{\dlt}  } = \alp\abs{ \inner{  f, \frac{ \partial x_{\alp}^{\dlt} }{ \partial \alp }  } }
$$
with $f\in \Ran \kl{\vk\kl{  T^* T  }}$. Therefore, in view of~\eqref{n5:10}, \eqref{n5:11},
for a given $f$, say $f = x_{\alp_i}^{\dlt}$, it is reasonable to use the following discretized
version of the quasi-optimality criterion in the linear functional strategy: choose
$\alp_i\kl{  y^{\dlt}  } = \alp_{  \kp_i  } $ from~\eqref{n5:7} such that
\begin{equation}\label{n5:12}
   \abs{ \inner{  x_{\alp_i}^{\dlt},  x_{\alp_{\kp_i}}^{\dlt} - x_{\alp_{\kp_i - 1 }}^{\dlt}  } } =
	 \min\skl{
	\abs{ \inner{  x_{\alp_i}^{\dlt},  x_{\alp_{j}}^{\dlt} - x_{\alp_{j - 1 }}^{\dlt}  } },\;
	j=2,3,\ldots,M
	}.
\end{equation}

To illustrate the quasi-optimality criterion in the aggregation~\eqref{n5:3}, \eqref{n5:12}, we simulate the data
by~\eqref{n1}, where $T$ is a matrix
$T =(t_{ij})$, where  $i=1,2,\ldots,m,$  $j=1,2,\ldots,n$ with the non-zero entries
$t_{kk} = a^k$, $0<a<1$, $x$ is a vector
$x = \kl{    x_{j} = j^{-\mu} \eta_j,\;  j=1,2,\ldots,n    }$, and $\eta_j$ are randomly sampled
from the uniform distribution on $[-1,1]$. We take $a=0.5$, $\mu=2$, $n=100$, $m=150$.

Our simulation mimics a severely ill-posed problem because the singular values
$\lm_k^2 = t_{kk}^2 = a^{2k}$ of $T^* T$ decrease exponentially, while the Fourier coefficients $x_j$ of $x$ in the corresponding
basis decrease only polynomially. A reason to consider this case is that, as it can be seen from Theorem~\ref{thm:main},
for severely ill-posed problems, the difference between the estimation of the solution and the functional estimation
is the most noticeable. For example, if $\vph(\lm) = \log^{-\nu} \frac{1}{\lm} $, $\nu>0$, which corresponds to the
severely ill-posed case, then the quasi-optimality criterion can guarantee an accuracy of order
$O\kl{   \log^{-\nu} \log  \frac{1}{\dlt}   }$ for an approximation of $x$, while the value of a bounded linear
functional $\inner{f, x  }$ can be estimated with the use of the quasi-optimality criterion much more accurately,
say with the accuracy of order
$   O\kl{    \dlt^{  2\gm^2  }   \log^{  -\nu \kl{   1+\gm - 2\gm^2   }  }    \frac{1}{\dlt}      }   $   when
$f\in  \Ran \kl{   \kl{  T^* T  }^{\gm} }     $,     $ 0 < \gm < 1/2 $.

%%%%%%%%%%%%%%%

Numerical illustrations below demonstrate that in the considered simulation scenario, the 
aggregation~\eqref{n5:3}, \eqref{n5:12},
which is based on the quasi-optimality criterion and the linear functional strategy, improves the accuracy resulting from the
quasi-optimality criterion and performs at the level of the best (but unknown) regularization parameter choice.

To guarantee almost surely that the Muckenhoupt-type condition~\eqref{eq:main_assumption_kappa} on the noise
$\xi$ is satisfied in our test, we simulate $\xi$ as
$\xi = \kl{   \xi_i,\; i=1,2,\ldots,m   }$, $m=150$, where $\xi_i$ are randomly sampled from the uniform distribution
on $   [-1,-\dlt] \cup [\dlt,1]   $, $\dlt > 0$, such that the noise support is separated from $0$ and $\infty$, as it is suggested
in Remark~\ref{rem1} discussed in the previous section.

The random simulations of $\xi$ and $x$ are performed $10$ times, and the noise intensity is chosen as $\dlt = 0.01$.
The regularized approximants $x_{\alp_i}^{\dlt}$ are constructed by the Tikhonov regularization, i.e.
$$
  x_{\alp_i}^{\dlt} = \kl{  \alp_i I + T^* T  }^{-1} T^* y^{\dlt},
$$
where $\alp_i$ are taken from~\eqref{n5:7} with $\alp_1 = 0.1$, $q=0.5$, $M=20$. Moreover, in each simulation,
the quasi-optimal regularization parameters $\alp = \alp\kl{  y^{\dlt}  }$, $\alp = \bar{\alp}\kl{  y^{\dlt}  }$
are chosen according to~\eqref{n5:8}, \eqref{n5:9}.
To guarantee condition~\eqref{n5:5}, we aggregate in~\eqref{n5:3} the regularized approximants
$x_{\alp_i}^{\dlt}$ with $\alp_i \geq \alp\kl{ y^{\dlt} }$. An aggregation on a wider set of approximants
does not improve the accuracy, as it has been observed.

%%%%%%%%%%%%%%%

The performance of the regularized approximants is measured in terms of the following quantities:
\begin{align*}
e_{\tm{qo}} &= \norm{ x - x_{ \alp\kl{  y^{\dlt}  }  }^{\dlt} },\quad
e_{\tm{qo,2}} = \norm{ x - \bar{x}_{ \bar{\alp}\kl{  y^{\dlt}  }  }^{\dlt} },
\\
e_{\tm{best}} &= \min\skl{ \norm{ x - x_{\alp_i}^{\dlt} },\;  i=1,2,\ldots,M },
\\
e_{\tm{best,2}} &= \min\skl{ \norm{ x - \bar{x}_{\alp_i}^{\dlt} },\;  i=1,2,\ldots,M },
\\
e_{\tm{agg}} &= \norm{  x - x_{ \tm{agg}, y^{\dlt} }^{s}  },
\end{align*}
where $s = \max\Set{   i:\; \alp_i \geq \alp\kl{ y^{\dlt} }   }$, and $x_{ \tm{agg}, y^{\dlt} }^{s}$ is given
by~\eqref{n5:3}, \eqref{n5:12}.
The mean values of the considered quantities over the performed simulations are given in Table~\ref{T1}.
The table also reports the values observed in a particular simulation displayed in Figure~\ref{F1}.

The presented illustration confirms that for severely ill-posed problems, the aggregation based on the linear functional
strategy is able to perform at the level of the best, but unknown, regularization parameter choice.

%%%%%%%%%%%%%%%

%
\begin{table}[t]

\begin{center}
\caption{Performance in terms of errors.
\label{T1}} \vspace{2mm}

\begin{tabular}{ccc}
\hline
error & mean value &  simulation on Figure~\ref{F1}  \\ \hline
$e_{\tm{qo}}$ & 0.079 & 0.076  \\
$e_{\tm{best}}$ & 0.067 & 0.064  \\
$e_{\tm{qo,2}}$ & 0.075 & 0.060  \\
$e_{\tm{best,2}}$ & 0.065 & 0.058  \\
$e_{\tm{agg}}$ & 0.065 & 0.057 
\\ \hline
\end{tabular}
\end{center}

\end{table}
%

%%%%%%%%%%%%%%%

%
\begin{figure}
\begin{center}
    \includegraphics[width=10cm]{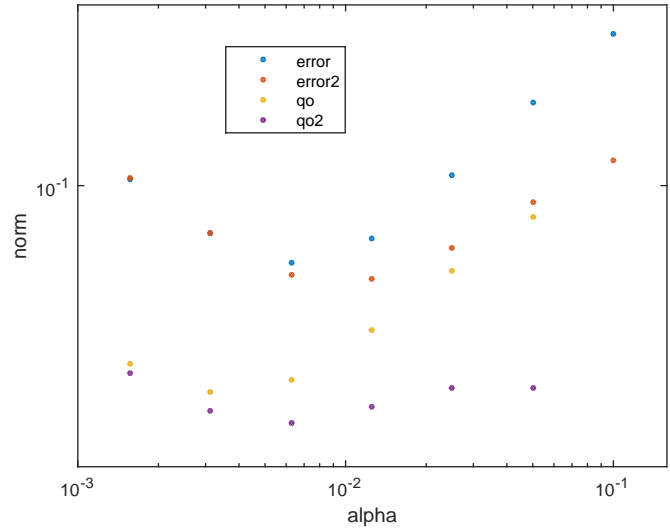}
\end{center}
\caption{
The quantities observed in a particular simulation: $\norm{ x - x_{\alp_i}^{\dlt} }$ (error), $\norm{ x - \bar{x}_{\alp_i}^{\dlt} }$ (error2),
$\norm{ x_{\alp_i}^{\dlt} - x_{\alp_{i-1}  }^{\dlt} }$ (qo),    $\norm{   \bar{x}_{\alp_i}^{\dlt} -     \bar{x}_{\alp_{i-1}  }^{\dlt} }$ (qo2),
plotted against the corresponding values of $\alp_i$, $i=1,2,\ldots,7$. \label{F1}}
\end{figure}
%

%%%%%%%%%%%%%%%
%%%%%%%%%%%%%%%
%%%%%%%%%%%%%%%
%%%%%%%%%%%%%%%

%%%%%%%%%%%%%%%%%%%%%%%%%%%%%%%%%%%%%%%%%%%%%%%%%%%%%%%%%%%%%%%%%%%%%%%%%%%%%%%%%%%%%%%%%%%%%%%%%%%%%%%%%%%%%%%%%%%%%%%%%%%%%%%%%

\section*{Acknowledgements}
This research was partially supported by AMMODIT project 645672 (Approximation Methods for Molecular Modelling and Diagnosis Tools)
in the frame of Horizon 2020 program.
%%%
Sergiy Pereverzyev Jr. gratefully acknowledges the support of the Austrian Science Fund (FWF): project P 29514-N32.
Stefan Kindermann is supported by the Austrian Science Fund (FWF) project
P 30157-N31.

%%%%%%%%%%%%%%%%%%%%%%%%%%%%%%%%%%%%%%%%%%%%%%%%%%%%%%%%%%%%%%%%%%%%%%%%%%%%%%%%%%%%%%%%%%%%%%%%%%%%%%%%%%%%%%%%%%%%%%%%%%%%%%%%%

%\bibliography{refs}

\nocite{*}

\begin{thebibliography}{10}

\bibitem{And86}
R.~Anderssen.
\newblock The linear functional strategy for improperly posed problems.
\newblock In J.~R. Cannon and U.~Hornung, editors, {\em Inverse Problems},
  volume~77 of {\em International Series of Numerical Mathematics}, pages
  11--30. {Birkhäuser Basel}, 1986.

\bibitem{And80}
R.~S. Anderssen.
\newblock On the use of linear functionals for {Abel}-type integral equations
  in applications.
\newblock In F.~{De Hoog} and M.~A. Lukas, editors, {\em The application and
  numerical solution of integral equations}, pages 195--221. Sijthoff and
  Noordhof International Publishers, 1980.

\bibitem{AndEng91}
R.~S. Anderssen and H.~W. Engl.
\newblock The role of linear functionals in improving convergence rates for
  parameter identification via {Tikhonov} regularization.
\newblock In M.~{Yamaguti et al.}, editor, {\em Inverse Problems in Engineering
  Sciences, ICM-90, Satellite Conference Proceedings}, pages 1--10. Springer,
  1991.

\bibitem{Bak84}
A.~B. Bakushinskii.
\newblock Remarks on choosing regularization parameter using the
  quasi-optimality and ratio criterion.
\newblock {\em USSR Comp. Math. Math. Phys.}, 24:181--182, 1984.

\bibitem{BauMatPer07}
F.~Bauer, P.~Math{\'e}, and S.~Pereverzev.
\newblock Local solutions to inverse problems in geodesy.
\newblock {\em J. Geodesy}, 81(1):39--51, 2007.

\bibitem{BauRei08}
F.~Bauer and M.~Rei{\ss}.
\newblock Regularization independent of the noise level: an analysis of
  quasi-optimality.
\newblock {\em Inverse Probl.}, 24(5):055009, 2008.

\bibitem{Bec11}
S.~M.~A. Becker.
\newblock Regularization of statistical inverse problems and the {Bakushinskii}
  veto.
\newblock {\em Inverse Probl.}, 27(11):115010, 2011.

\bibitem{BelKasCas72}
R.~Bellman, B.~G. Kashef, and J.~Casti.
\newblock Differential quadrature: a technique for the rapid solution of
  nonlinear partial differential equations.
\newblock {\em J. Comput. Phys.}, 10(1):40--52, 1972.

\bibitem{ChePerXu15}
J.~Chen, S.~Pereverzyev~Jr., and Y.~Xu.
\newblock Aggregation of regularized solutions from multiple observation
  models.
\newblock {\em Inverse Probl.}, 31(7):075005, 2015.

\bibitem{EngNeu88}
H.~W. Engl and A.~Neubauer.
\newblock A parameter choice strategy for (iterated) {Tikhonov} regularization
  of ill-posed problems leading to superconvergence with optimal rates.
\newblock {\em Appl. Anal.}, 27:5--18, 1988.

\bibitem{GolPer00}
A.~Goldenshluger and S.~V. Pereverzev.
\newblock Adaptive estimation of linear functionals in {Hilbert} scales from
  indirect white noise observations.
\newblock {\em Probab. Theory Related Fields}, 118(2):169--186, 2000.

\bibitem{HaPaRe11}
U.~H{\"a}marik, R.~Palm, and T.~Raus.
\newblock Comparison of parameter choices in regularization algorithms in case
  of different information about noise level.
\newblock {\em Calcolo}, 48(1):47--59, 2011.

\bibitem{KinNeu08}
S.~Kindermann and A.~Neubauer.
\newblock On the convergence of the quasioptimality criterion for (iterated)
  {Tikhonov} regularization.
\newblock {\em Inverse Probl. Imaging}, 2(2):291--299, 2008.

\bibitem{KusKle02}
J.~Kusche and R.~Klees.
\newblock Regularization of gravity field estimation from satellite gravity
  gradients.
\newblock {\em J. Geodesy}, 76(6):359--368, 2002.

\bibitem{Leo91}
A.~S. Leonov.
\newblock On the accuracy of {Tikhonov} regularizing algorithms and
  quasioptimal selection of a regularization parameter.
\newblock {\em Soviet Math. Dokl.}, 44:711--716, 1991.

\bibitem{LouMaa90}
A.~K. Louis and P.~Maass.
\newblock A mollifier method for linear operator equations of the first kind.
\newblock {\em Inverse Probl.}, 6(3):427--440, 1990.

\bibitem{LuPer13}
S.~Lu and S.~V. Pereverzev.
\newblock {\em Regularization theory for ill-posed problems: selected topics}.
\newblock Walter de Gruyter, 2013.

\bibitem{MatPer02}
P.~Math{\'e} and S.~V. Pereverzev.
\newblock Direct estimation of linear functionals from indirect noisy
  observations.
\newblock {\em J. Complexity}, 18(2):500--516, 2002.

\bibitem{Neu08}
A.~Neubauer.
\newblock The convergence of a new heuristic parameter selection criterion for
  general regularization methods.
\newblock {\em Inverse Probl.}, 24(5):055005, 2008.

\bibitem{Palm10}
R.~Palm.
\newblock {\em Numerical Comparison of Regularization Algorithms for Solving
  Ill-Posed Problems}.
\newblock PhD thesis, Institute of Computer Science, University of Tartu, 2010.

\bibitem{TikGla65}
A.~N. Tikhonov and V.~B. Glasko.
\newblock Use of the regularization method in non-linear problems.
\newblock {\em USSR Comp. Math. Math. Phys.}, 5:93--107, 1965.

\end{thebibliography}
%\bibliographystyle{abbrv}

\end{document}